\ifdefined\journalstyle
  \documentclass[final,3p,times]{elsarticle}
\else
  \documentclass[11pt,a4paper]{article}
\fi

\usepackage{amssymb,amsmath,tabularx}

\usepackage{tikz}
\usetikzlibrary{arrows}

\usepackage{xcolor}
\usepackage[hidelinks,hypertexnames=false,hyperindex=true,pdfpagelabels,linktoc=all,colorlinks=true,linkcolor=teal,citecolor=teal,urlcolor=teal]{hyperref}

\ifdefined\journalstyle
\else
\fi

\newtheorem{theorem}{Theorem}
\newtheorem{lemma}{Lemma}
\newtheorem{corollary}{Corollary}

\newenvironment{proof}{\begin{trivlist}
\item[\hskip\labelsep{\it Proof.}]}{$\hfill\Box$\end{trivlist}}

\ifdefined\journalstyle
  \newcommand{\RefApp}[1]{\ref{#1}}
\else
  \newcommand{\RefApp}[1]{Appendix~\ref{#1}}
\fi

\newcommand{\U}{\mathit{U}}    
\newcommand{\SU}{\mathit{SU}}  

\newcommand{\satop}[2]{\stackrel{\scriptstyle{#1}}{\scriptstyle{#2}}}
\newcommand{\bsDelta}{\boldsymbol{\Delta}}

\newcommand{\bsv}{\boldsymbol{v}}

\newcommand{\bsx}{\boldsymbol{x}}
\newcommand{\bsh}{\boldsymbol{h}}

\newcommand{\bst}{\boldsymbol{t}}
\newcommand{\bsz}{\boldsymbol{z}}
\newcommand{\bsu}{\boldsymbol{u}}

\newcommand{\bsy}{\boldsymbol{y}}

\newcommand{\bsphi}{\boldsymbol{\phi}}
\newcommand{\bszero}{\boldsymbol{0}}
\newcommand{\rd}{\mathrm{d}}
\newcommand{\ri}{\mathrm{i}}
\newcommand{\calH}{\mathcal{H}}
\newcommand{\calI}{{\mathcal{I}}}

\newcommand{\calO}{{\mathcal{O}}}
\newcommand{\calQ}{{\mathcal{Q}}}

\newcommand{\bbZ}{{\mathbb{Z}}}
\newcommand{\bbN}{{\mathbb{N}}}
\newcommand{\bbR}{{\mathbb{R}}}
\newcommand{\Fourier}{{\mathfrak{F}}}

\begin{document}

\ifdefined\journalstyle
  \begin{frontmatter}
\fi

\title{Lattice Meets Lattice: \\ Application of Lattice Cubature to \\ Models in Lattice Gauge Theory}

\ifdefined\journalstyle  
  \author[1,1b]{Tobias Hartung}
  \ead{tobias.hartung@desy.de}
  \author[2]{Karl Jansen}
  \ead{karl.jansen@desy.de}
  \author[3]{Frances Y.~Kuo}
  \ead{f.kuo@unsw.edu.au}
  \author[4]{Hernan Le\"ovey}
  \ead{hernaneugenio.leoevey@axpo.com}
  \author[5]{Dirk Nuyens\corref{cor1}}
  \ead{dirk.nuyens@cs.kuleuven.be}
  \author[3]{Ian H.~Sloan}
  \ead{i.sloan@unsw.edu.au}
  \address[1]{Computation-based Science and Technology Research Center, The Cyprus Institute, 20 Konstantinou Kavafi Street, 2121 Nicosia, Cyprus}
  \address[1b]{Department of Mathematics, King's College London, Strand, London WC2R 2LS, United Kingdom}
  \address[2]{NIC, DESY Zeuthen, Platanenallee 6, 15738 Zeuthen, Germany}
  \address[3]{School of Mathematics and Statistics, UNSW Sydney, Sydney NSW 2052, Australia}
  \address[4]{Structured Energy Trading, AXPO Trading \& Sales, Parkstrasse 23, 5400 Baden, Germany}
  \address[5]{Department of Computer Science, KU Leuven, Celestijnenlaan 200A, 3001 Leuven, Belgium}
  \cortext[cor1]{Corresponding author}
\else
  \author{
     Tobias Hartung\footnote{%
       Computation-based Science and Technology Research Center, The Cyprus Institute, 20 Konstantinou Kavafi Street, 2121 Nicosia, Cyprus
       and
       Department of Mathematics, King's College London, Strand, London WC2R 2LS, United Kingdom
       (email: tobias.hartung@desy.de)},
     \quad
     Karl Jansen\footnote{%
       NIC, DESY Zeuthen, Platanenallee 6, 15738 Zeuthen, Germany
       (email: karl.jansen@desy.de)},
     \quad
     Frances Y.~Kuo\footnote{%
       School of Mathematics and Statistics, UNSW Sydney, Sydney NSW 2052, Australia
       (email: f.kuo@unsw.edu.au)},
     \\
     Hernan Le\"ovey\footnote{%
       Structured Energy Trading, AXPO Trading \& Sales, Parkstrasse 23, 5400 Baden, Germany
       (email: hernaneugenio.leoevey@axpo.com)},
     \quad
     Dirk Nuyens\footnote{%
       Department of Computer Science, KU Leuven, Celestijnenlaan 200A, 3001 Leuven, Belgium
       (email: dirk.nuyens@cs.kuleuven.be)},
     \quad
     Ian H.~Sloan\footnote{%
       School of Mathematics and Statistics, UNSW Sydney, Sydney NSW 2052, Australia
       (email: i.sloan@unsw.edu.au)}}
\fi

\date{March 2021}

\ifdefined\journalstyle
\else
  \maketitle
\fi

\begin{abstract}
High dimensional integrals are abundant in many fields of research
including quantum physics. The aim of this paper is to develop efficient
recursive strategies to tackle a class of high dimensional integrals
having a special product structure with \emph{low order couplings},
motivated by models in \emph{lattice gauge theory} from quantum field
theory. A novel element of this work is the potential benefit in using
\emph{lattice cubature rules}. The group structure within lattice rules
combined with the special structure in the physics integrands may allow
efficient computations based on Fast Fourier Transforms. Applications to
the quantum mechanical rotor and compact $\U(1)$ lattice gauge theory in
two and three dimensions are considered.
\end{abstract}

\ifdefined\journalstyle
\fi

\ifdefined\journalstyle
  \end{frontmatter}
\fi

\section{Introduction} \label{sec:intro}

High dimensional integrals are abundant in many fields of research
including quantum physics.
The aim of this paper is to develop efficient recursive
strategies to tackle a class of high dimensional integrals having a
special product structure with \emph{low order couplings}, motivated by
models in \emph{lattice gauge theory} \cite{Gattringer:2010zz} from
quantum field theory. A novel element of this work is the potential
benefit in using \emph{lattice cubature rules} \cite{DKS13,SJ94}. The
group structure within lattice rules combined with the special structure
in the physics integrands may allow efficient computations based on Fast
Fourier Transforms (FFT). Note the two different occurrences of the word
``\emph{lattice}'' here: one as a discretization tool in physics, and one
as a method of approximating integrals in numerical analysis.

The integrals being considered are motivated by problems from simulations
of quantum field theories, e.g., systems in high energy and statistical
physics. In the simplest formulation, we have an $L$-dimensional integral
of the form
\begin{align} \label{eq:int1}
 \int_{D^L} \prod_{i=0}^{L-1} f_i\big(x_i,x_{i+1}\big) \,\rd\bsx,
 \qquad\mbox{with $\bsx = (x_0,\ldots,x_{L-1})$ and $x_L\equiv x_0$,}
\end{align}
where each variable $x_i$ belongs to a bounded domain $D\subset \bbR$ and
each function $f_i$ depends only on two consecutive variables $x_i$ and
$x_{i+1}$. We will refer to the condition $x_L\equiv x_0$ as
``\emph{parametric periodicity}''. This is a defining characteristic of
the class of problems that we consider in this paper. Models in lattice
field theory often couple the relevant degrees of freedom on neighboring
lattice sites, and as a result exhibit the structure of \eqref{eq:int1}.
As a concrete example we consider here the \emph{topological oscillator}
or \emph{quantum rotor}, see
\cite{AGHJLV16,Bietenholz:1997kr,Bietenholz:2010xg} and
Section~\ref{sec:rotor} below, which is a 1D (i.e., one space-time
dimension) quantum mechanical model in Euclidean time with $L$ lattice
points. Often the functions $f_i$ are identical (or only one of them is
different due to the presence of the ``observable function''), depend only
on the difference of the two input variables, and are periodic in each
coordinate direction. For example, $f_i(u,v) = \exp(\beta\cos(v-u))$ with
$D = [-\pi,\pi)$.

Note that there are two senses of dimensionality here: typically by
dimension we are referring to the \emph{number of integration variables},
denoted by~$L$ in \eqref{eq:int1} and later more generally denoted by~$s$;
but there is also the \emph{space-time dimension} of the underlying
physical problem which we will denote by $d$ and describe as 1D, 2D or 3D
(i.e., $d=1,2,3,\ldots$).

Often in high energy and statistical physics, integrals of the form
\eqref{eq:int1} appear in both the numerator and denominator of a ratio
that represents the expected value of an observable (see e.g.,
\eqref{eq:ratio} below). The end goal is typically to obtain this ratio
rather than the separate integrals, and a popular strategy is to tackle
this using \emph{Markov chain Monte Carlo $($MCMC$)$} simulations. Here
instead we propose to treat the integrals by numerical integration
methods.

One possibility is to apply an $L$-fold tensor product of a
one-dimensional quadrature rule to approximate the integral
\eqref{eq:int1}. That is, each integral over $D$ in \eqref{eq:int1} is
approximated by a one-dimensional quadrature rule with $n$ points
$t_1,\ldots,t_n\in D$ and corresponding weights $w_1,\ldots, w_n \in\bbR$,
so that the approximation to \eqref{eq:int1} is given by
\begin{align} \label{eq:prod1}
  \sum_{k_0=0}^{n-1} w_{k_0} \cdots \sum_{k_{L-1}=0}^{n-1} w_{k_{L-1}}\,
 \prod_{i=0}^{L-1} f_i\big(t_{k_i},t_{k_{i+1}}\big),
\end{align}
where for convenience we extend parametric periodicity to the notation of
indices so that $k_L \equiv k_0$.

We switch now to the general notation with $s$ instead of $L$ denoting the
number of integration variables. The approximation of \eqref{eq:int1} by
\eqref{eq:prod1} can be expressed in the general form
\begin{align} \label{eq:gen}
  \int_{D^s} f(\bsx) \,\rd\bsx
  \,\approx\,
  \sum_{k=1}^{N} \omega_k f(\bst_k),
\end{align}
where the $N = n^s$ points are denoted by $\bst_1,\ldots,\bst_N\in D^s$,
with the corresponding weights $\omega_1,\ldots,\omega_N\in\bbR$ being
products of the one-dimensional weights.
Product rules are generally \emph{not} recommended in high dimensions
because the computational cost is generally $\calO(n^s) = \calO(N)$. If
the one-dimensional rule with a general integrand has error
$\calO(n^{-\alpha})$ for some $\alpha>0$, then the error for the product
rule is also $\calO(n^{-\alpha})$. Expressing now the error with respect
to the cost, it would be $\calO(N^{-\alpha/s})$, which suffers from the
\emph{curse of dimensionality} for large $s$.

Some alternatives to product rules in high dimensions include \emph{Monte
Carlo $($MC$)$ methods}, \emph{quasi-Monte Carlo $($QMC$)$ methods}
\cite{DP10,DKS13,Hic98b,Lem09,Nie92,SJ94}, and \emph{sparse grid $($SG$)$
methods} \cite{BG04}. SG methods use only a strategically chosen sparse
subset of the $n^s$ grid points, thus reducing the cost. Both MC and QMC
methods take the form of \eqref{eq:gen} with equal weights $\omega_k =
1/N$, but now $N$ is no longer the $s$th power of some number $n$, thus
avoiding the exponential cost in $s$. MC and QMC methods are fundamentally
different: the MC points are generated randomly and have an error
convergence rate of $\calO(N^{-1/2})$, while the QMC points are designed
and chosen deterministically to be better than random in the sense of
achieving a higher order convergence rate $\calO(N^{-\alpha})$, with
$\alpha$ close to $1$ or better, where $\alpha$ depends on the smoothness
of the integrand and properties of the QMC points. Modern QMC analysis in
\emph{weighted function spaces} can give error bounds that are
\emph{independent of $s$}, see e.g., \cite{DKS13}. Conceptually, this
requires that the integrands have \emph{low effective dimension}, either
in the sense that only the initial variables are important (\emph{low
truncation dimension}), or that an integrand is dominated by a sum of
terms involving only a few variables at a time (\emph{low superposition
dimension}), see e.g., \cite{CMO97}. QMC methods have also been considered
for physics applications, e.g., in \cite{GHJLGM14,JLAGM14} for lattice
systems and in \cite{Borowka:2018goh,deDoncker:2018nqe} for multiloop
calculations in perturbation theory.

However, integrands of the form in \eqref{eq:int1} most likely will
\emph{not} have low effective dimension in either of the two senses
described above. So it is unclear if a direct application of QMC methods
to \eqref{eq:int1} would be fruitful. On the other hand, these integrands
have a special structure: since each product factor in \eqref{eq:int1}
depends only on two neighboring variables, it has been shown in
\cite{AGHJLV16} that a recursive integration strategy can be used to
evaluate the product rule \eqref{eq:prod1} very efficiently without an
exponential cost in $s=L$. We refer to \eqref{eq:int1} as an example in 1D
with \emph{first order coupling}, and denote this coupling order by $r=1$;
see Figure~\ref{fig1}(a) for an illustration of the active variables in
$f_i$. Recursive integration has been considered in e.g.,
\cite{Cra08,GenKal86,Hay06,Hay11}, but not for integrands with parametric
periodicity $x_L \equiv x_0$. For parametric periodicity a more careful
analysis is needed.

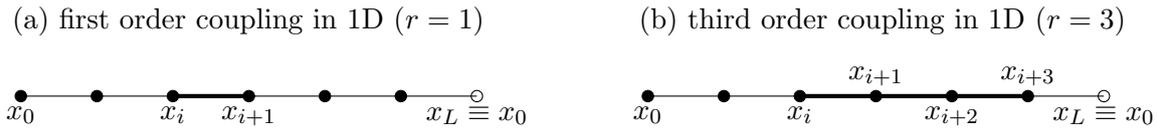
\begin{figure}
  \centering
  \begin{tikzpicture}[scale=1]
    \begin{scope}
      \node[anchor=west] at (-0.25,1) {(a) first order coupling in 1D $(r=1)$};
      \foreach \i in {0,...,5} {
        \filldraw[black] (\i,0) circle (.075);
      }
      \draw (6,0) circle (.075);
      \node[below] at (0,0) {$x_0$};
      \node[below] at (2,0) {$x_i$};
      \node[below] at (3,0) {$x_{i+1}$};
      \node[below] at (6,0) {$x_L \equiv x_0$};
      \draw (0,0) -- (6,0);
      \draw[ultra thick] (2,0) -- (3,0);
    \end{scope}
    \begin{scope}[shift={(8.25,0)}]
      \node[anchor=west] at (-0.25,1) {(b) third order coupling in 1D $(r=3)$};
      \foreach \i in {0,...,5} {
        \filldraw[black] (\i,0) circle (.075);
      }
      \draw (6,0) circle (.075);
      \node[below] at (0,0) {$x_0$};
      \node[below] at (2,0) {$x_i$};
      \node[above] at (3,0) {$x_{i+1}$};
      \node[below] at (4,0) {$x_{i+2}$};
      \node[above] at (5,0) {$x_{i+3}$};
      \node[below] at (6,0) {$x_L \equiv x_0$};
      \draw (0,0) -- (6,0);
      \draw[ultra thick] (2,0) -- (5,0);
    \end{scope}
  \end{tikzpicture}
\caption{Illustration of variable couplings}\label{fig1}
\end{figure}

\smallskip

\textbf{The first contribution of this paper} is to review this recursive
strategy from~\cite{AGHJLV16,Hartung:2020uuj} and identify favorable
scenarios when the cost can be further reduced using FFT. Our findings are
summarized in Table~\ref{tab1} in Section~\ref{sec:first}. As an example
of the best scenario (Scenario (A7)), for the integral \eqref{eq:int1}
with $f_i(u,v) = \exp(\beta \cos(v-u))$, the cost for the product
rectangle rule is only
\[
  \calO(n\log(n)),
  \qquad\mbox{\emph{independently of $L$}},
\]
with an error of $\calO(n^{-\alpha})$ (or even converging exponentially
fast in $n$) due to the well-known Euler--Maclaurin formula for
trapezoidal rules applied to periodic functions of smoothness order
$\alpha$. In other words, it is possible to achieve the error of a full
tensor product rule with a cost that is independent of the number of
integration variables $L$. This is a remarkable outcome!

\smallskip

\textbf{The second contribution of this paper} is to extend the recursive
strategy to the situation where the domain $D$ in \eqref{eq:int1} is
replaced by an $s$-dimensional domain $D^s$, giving an $L$-fold product of
$s$-dimensional integrals of the form
\begin{align} \label{eq:intLsb}
 \int_{(D^s)^L} \prod_{i=0}^{L-1} f_i\big(\bsx_i,\bsx_{i+1}\big) \,\rd\bsx.
\end{align}
Here $\bsx \in (D^s)^L = D^{L\times s}$ can be interpreted as a matrix of
size $L\times s$, with $\bsx_i = (x_{i,0},\ldots,x_{i,s-1}) \in D^s$
referring to its $i$th row, and we have parametric periodicity $x_{i,j}
\equiv x_{i\bmod L,\; j\bmod s}$ for all $i,j\in\bbN$. Compared to our
strategy for $s=1$ we here just need to replace the $1$-dimensional
quadrature rules by $s$-dimensional cubature rules. In particular, we
propose to use $n$-point \emph{lattice cubature rules}, see below. We
obtain completely analogous results in Table~\ref{tab1}. In the best
scenario (Scenario (A7)) where FFT is applicable, the cost is only
\[
  \calO(n\log(n)),
  \qquad\mbox{\emph{independently of $L$ and $s$}}.
\]
The error is $\calO(n^{-\alpha})$ where $\alpha$ depends on the smoothness
of the integrand and properties of the underlying lattice rule. The
implied constant in the error bound is again independent of $L$, but can
potentially depend exponentially on $s$, unless the integrand belongs
to a suitable weighted function space, see the brief introduction of
lattice cubature rules near the end of the introduction and the references
there.

\smallskip

\textbf{The third contribution of this paper} is to extend the recursive
strategy further to higher order couplings with arbitrary order~$r$, that
is, each function $f_i(x_i,x_{i+1})$ in \eqref{eq:int1} is now replaced by
\[
 f_i(x_i,x_{i+1},\ldots,x_{i+r}),
\]
with parametric periodicity $x_i \equiv x_{i\bmod L}$ in place for all
indices $i\in\bbN$; see Figure~\ref{fig1}(b) for an illustration of the
active variables in $f_i$ with $r=3$. Our aim is to control the cost with
increasing~$r$. The motivation from physics models is that, instead of
approximating a first derivative (e.g., the angular velocity) by a forward
difference formula which depends on two neighboring variables, we use a
central difference formula or other higher order difference formulas which
depend on more neighboring variables. This is motivated by the fact that
in this way the discrete lattice effects are suppressed and one is working
closer to the desired continuum model. Here we propose to use a tensor
product of an $r$-dimensional lattice cubature rule (assuming for
simplicity that $r$ divides $L$). Our findings are summarized in
Table~\ref{tab2} in Section~\ref{sec:higher}. We show that under the right
condition (i.e., if the functions $f_i$ have certain desirable properties,
see Scenario (B7)), the cost of the recursive strategy based on a
$(L/r)$-fold product of an $r$-dimensional lattice rule with $n$ points is
only
\[
  \calO(n\log(n)),
  \qquad\mbox{\emph{independently of $L$ and $r$}}.
\]
The error is again $\calO(n^{-\alpha})$ where $\alpha$ depends on the
integrand and the lattice rule, and the implied constant is independent of
$L$ but can potentially depend exponentially on $r$. In case that the
desired condition does not hold, we have an alternative strategy (see
Scenario (B4)) with cost of $\calO(\log(L/r)\,n^3)$.

\smallskip

\textbf{The fourth contribution of this paper} is to generalize the
recursive strategy to 2D problems with first order couplings of the
generic form
\begin{align} \label{eq:int2D}
  \int_{D^{2L^2}} \prod_{i=0}^{L-1} \prod_{j=0}^{L-1}
  f_{i,j}\big(x^a_{i,j} - x^a_{i,j+1} - x^b_{i,j} + x^b_{i+1,j} \big)
  \,\rd\bsx,
\end{align}
where $\bsx = (\bsx^a,\bsx^b)$ has $2L^2$ components $x_{i,j}^a$ and
$x_{i,j}^b$ for $i,j=0,\ldots, L-1$, with parametric periodicity
$x_{i,j}^a \equiv x_{i\bmod L,\,j\bmod L}^a$ and  $x_{i,j}^b \equiv
x_{i\bmod L,\,j\bmod L}^b$ for all indices $i,j\in\bbN$.
Generally speaking, the form of the integrand given in \eqref{eq:int2D} is
characteristic of gauge theories which are at the heart of quantum field
theories to describe the mediating fields between matter fields, e.g.,
quarks. Here we consider a pure $\U(1)$ gauge theory in two dimensions,
so-called ``2-dimensional compact $\U(1)$ lattice gauge theory'', see
Section~\ref{sec:QED} below. We first show that the problem can be turned
into a nested integration problem where our favorable strategy (see
Scenario (A7)) can be applied to both the outer and inner integrals. Then
we show that the problem further simplifies to become essentially a 1D
problem, yielding an overall cost of only
\[
  \calO(n\log(n)),
  \qquad\mbox{\emph{independently of $L$}}.
\]
The error is $\calO(n^{-\alpha})$ with $\alpha$ depending on the
smoothness of the functions. This is truly remarkable for an integration
problem with $2L^2$ variables.

\smallskip

\textbf{The future outlook} is to generalize the recursive strategy to 2D
problems with higher order couplings, and also to 3D problems where the
number of integration variables is $3L^3$, and eventually couple in matter
fields. We have made some progress in this direction. The challenge is
that, even with the recursive strategy reducing the dimensionality of the
problems, the remaining integrals to be evaluated are still very high
dimensional. We derived some explicit representations of the integrals in
terms of Fourier coefficients of the functions, and these may hold the key
to tackle these tough problems in future work.

\smallskip

Up to this point we have not yet formally introduced lattice cubature
rules; we will do this now. \emph{Rank-$1$ lattice rules}
\cite{DKS13,Hic98b,Lem09,SJ94} are a family of QMC methods for
approximating an integral over the $s$-dimensional unit cube $[0,1]^s$,
with $s=L$ or $r$ or else as appropriate. They take the form
\begin{align} \label{eq:lat}
  \frac{1}{N} \sum_{k=1}^{N} f\bigg(\frac{k\bsz\bmod N}{N}\bigg),
\end{align}
where $\bsz \in\bbZ^s$ is known as the \emph{generating vector} of the
lattice rule and determines the quality of the rule. There is an
underlying \emph{infinite lattice} of points $\{\frac{k\bsz}{N} :
k\in\bbZ\} \subset \bbR^s$. The sum or difference of any two lattice
points is another lattice point. The cubature points of a lattice rule are
those points from the infinite lattice that lie in the half-open unit
cube, together with the origin. They form a \emph{group} under addition
modulo the integers. If a given integral is not formulated over the unit
cube, then an appropriate change of variables should be used to
reformulate the problem.

Loosely speaking, lattice rules can achieve the error $\calO(N^{-\alpha})$
if the integrand $f$ is periodic with respect to each variable, and if its
Fourier coefficients decay in a suitable way, where $\alpha$ is a
smoothness parameter which roughly corresponds to the number of available
mixed derivatives of $f$, see e.g., \cite{SJ94}. The implied constant in
the error bound for lattice rules can depend exponentially on $s$, but can
also be independent of~$s$ by working in a weighted function space setting
with sufficiently decaying weight parameters. Good lattice generating
vectors can be obtained by \emph{fast component-by-component
constructions}, see e.g., \cite{DKS13,Nuy14}.

The outline of the paper is as follows. In Section~\ref{sec:physics} we
introduce various physics models of interest, to motivate the special
product structure of the integrands that we consider in this paper. For
the next two sections, we move away from the explicit physics models and
instead consider generic classes of integrands with specific product
structures, and develop strategies to approximate the integrals
efficiently under various assumptions on the properties of the product
factors. Specifically, in Section~\ref{sec:first} we review and extend the
recursive strategy from \cite{AGHJLV16} in detail, including the use of
FFT and the extension to an $L$-fold product of $s$-dimensional domains;
in Section~\ref{sec:higher} we extend the recursive strategy to higher
order couplings. Then in Section~\ref{sec:app-rotor} we return to the
application of the recursive strategy to the quantum rotor, and in
Section~\ref{sec:app-QED} we consider applications to compact $\U(1)$
lattice gauge theory. Section~\ref{sec:summary} provides a brief
summary. The Appendix includes derivations of explicit expressions for
some integrals using Fourier series, paving ways for future work.

\section{Description of physics models} \label{sec:physics}

Although in this article we are mainly interested in the mathematical
structure of the integrals to be solved, the motivation for the form of
the integrals come from physics models. We therefore describe here two
models from physics which lead to such structures.

The characteristic of a quantum mechanical system is that not only the
classical path contributes to a physical observable, but -- according to
Feynman -- that all possible paths have to be taken into account. This
leads to the notion of a {\em path integral}. Following Feynman's
description, a quantum mechanical system is defined by the path integral
\begin{align}
  \int_{D^L} \exp(-S[\bsphi]) \,\rd\bsphi , \label{equ:pint}
\end{align}
where $D$ is a domain in $\bbR$ and $S[\bsphi]$ denotes the {\em action}
of the considered quantum mechanical model. For a general, mathematically
well-defined formulation of the path integral we refer to
\cite{Hartung:2018usn,Jansen:2019uei}. Here we consider the special case
of a time-discretized quantum mechanical system which also renders the
path integral well defined. However, the integral in \eqref{equ:pint} can
still be of high dimension $L$, where $L$ can be in the thousands or much
larger. Of physical interest is the expected value $\langle
O[\bsphi]\rangle$ of an observable $O[\bsphi]$ which can be calculated
within the path integral formalism as
\begin{align}
  \langle O[\bsphi] \rangle \,=\, \frac
  {\int_{D^L} O[\bsphi] \, \exp(-S[\bsphi]) \,\rd\bsphi}
  {\int_{D^L} \exp(-S[\bsphi]) \,\rd\bsphi} . \label{eq:ratio}
\end{align}

\subsection{Quantum rotor} \label{sec:rotor}

As anticipated in the introduction, the first model we are going to
consider in this article is the quantum rotor which describes a particle
with mass $m_0$ moving on a circle with radius~$r_0$, see
\cite{AGHJLV16,Bietenholz:1997kr,Bietenholz:2010xg}. Thus, the particle
has a moment of inertia $I = m_0 r_0^2$. We investigate this particular
model because it has already some characteristic features of non-linear
$\sigma$-models and gauge theories, see e.g., \cite{Gattringer:2010zz}. An
example of such a gauge theory will be discussed in the next subsection.

In the simple quantum rotor model, the free coordinate of the system is
the angle $\phi\in D=[-\pi,\pi)$ describing the position of the particle
on the circle. In the continuum, the system is described by the action
\begin{align} \label{eq:limit}
 S(\phi)=\int_{0}^T \frac{I }{2}\Big(\frac{\rd \phi}{\rd t}\Big)^2 \rd t.
\end{align}
With $T$ and $I$ kept fixed, taking $L$ time discretizations with lattice
spacing $h = T/L$, and approximating $\frac{1}{2}(\frac{\rd \phi}{\rd
t})^2 \approx \frac{1}{2}(\frac{\phi_{i+1} - \phi_i}{h})^2 \approx \frac{1
- \cos(\phi_{i+1} - \phi_i)}{h^2}$, we obtain the discretized action and
observables
\begin{align} \label{eq:one}
  S[\bsphi]
  \,=\, \frac{I}{h^2}\sum_{i=0}^{L-1} \big(1-\cos(\phi_{i+1}-\phi_{i})\big)
  \qquad\mbox{and}\qquad
  O[\bsphi] \,=\, \cos(\phi_{k+1} - \phi_k) \quad\mbox{for any $k$},
\end{align}
where $\phi_L \equiv \phi_0$, i.e., we assume periodic boundary
conditions. Although unusual at first sight, the choice of approximation
$\frac{1}{2}(\frac{\rd \phi}{\rd t})^2 \approx \frac{1 - \cos(\phi_{i+1} -
\phi_i)}{h^2}$ is important, since the cosine introduces periodicity.
Numerically this will allow us to use FFT, and so reduce the computational
cost of the integration problem significantly. From a physics point of
view, the cosine form is also interesting because it resembles actions
used for gauge theories and thus provides a proving ground before tackling
gauge theories.

Furthermore, it is important to note that any constant term in the action
$S$, such as the~$1$ in $1-\cos$, can be removed because the action enters
the ratio $\frac{\int O\, e^{-S}}{\int e^{-S}}$ only as an exponent in
both numerator and denominator. Any constant contributions to the action
will therefore cancel. Hence, the ratio \eqref{eq:ratio} is given by
\begin{align} \label{eq:rotor}
  \langle O[\bsphi] \rangle \,=\, \frac
  {\int_{D^L} \cos\big(\phi_{k+1} - \phi_k) \, \exp(\beta\sum_{i=0}^{L-1}\cos(\phi_{i+1}-\phi_{i})\big) \,\rd\bsphi}
  {\int_{D^L} \exp\big(\beta\sum_{i=0}^{L-1}\cos(\phi_{i+1}-\phi_{i})\big) \,\rd\bsphi},
  \qquad \beta = \frac{IL^2}{T^2}.
\end{align}

Note that the arguments leading to the continuum action \eqref{eq:limit}
is referred to as the ``naive continuum limit'' since the true,
non-perturbative continuum limit is reached by sending $h\rightarrow 0$
and therefore $L = T/h\to\infty$, while keeping $T$ and $I$ fixed. The
form of the lattice action in \eqref{eq:rotor} is not unique. Any action
that reproduces the continuum time derivative when $h \to 0$ and
$L\to\infty$ is a valid discretization of the continuum action. For
example, we may use a higher order finite difference formula instead of
forward difference, leading to higher order couplings in the variables.
More complicated forms of the lattice action can have the advantage that
unwanted lattice spacing effects are canceled out. Whether such possible,
more complicated forms of an action reproduce the correct continuum
physics can, however, only be answered when the above sketched
non-perturbative procedure of the continuum limit is carried through.

Conversely, this non-uniqueness of the action can also be used to our
advantage. Above, we have introduced $\frac{1}{2}(\frac{\rd \phi}{\rd
t})^2 \approx \frac{1 - \cos(\phi_{i+1} - \phi_i)}{h^2}$ instead of the
canonical $\frac{1}{2}(\frac{\rd \phi}{\rd t})^2 \approx
\frac{1}{2}(\frac{\phi_{i+1} - \phi_i}{h})^2$. Both choices reproduce the
correct continuum action~\cite{AGHJLV16}, yet the cosine choice has
significant numerical advantages. In this sense, it is not only important
to develop numerical methods that are well-adjusted to meet the needs of
computational physics, but modeling in computational physics is an
important step which allows for lattice actions to be designed such that
they can be addressed efficiently with existing numerical methods. A prime
example is the discretization of fermionic actions which come in a number
of different incarnations with different advantages and disadvantages
\cite{Gattringer:2010zz,Jansen:2008vs}.

Observe that both $S[\bsphi]$ and $O[\bsphi]$ are $2\pi$-periodic with
respect to each of the integration variables $\phi_i$. So the integrals
remain unchanged if we shift the integration domain from $[-\pi,\pi)$ to
$[0,2\pi)$. We can then apply the linear mapping $\phi = 2\pi x$ in each
coordinate to convert the integrals into the unit cube $[0,1)^L$.

We can also consider a slightly more complicated observable called the
topological susceptibility. The topological susceptibility $\chi_t$ is
related to the topological charge $Q$ of the system
\begin{align} \label{eq:topcharge}
 Q[\bsphi]
 \,=\,
 \frac1{2\pi} \sum_{i=0}^{L-1} \Big(
 (\phi_{i+1}-\phi_{i}) \bmod [-\pi,\pi) \Big),
\end{align}
where
\[
  \psi \bmod [-\pi,\pi)
  \,:=\,
  \begin{cases}
  \psi \bmod 2\pi & \mbox{if } \psi \bmod 2\pi \in [0,\pi), \\
  \psi \bmod 2\pi - 2\pi & \mbox{otherwise.}
  \end{cases}
\]
The topological susceptibility is then given by the expectation (with $h$
set to $1$)
\begin{align}\label{eq:topsusc}
  \chi_t \,=\, \frac{1}{L} \big\langle (Q[\bsphi])^2 \big\rangle.
\end{align}
In lay terms, the topological charge captures the winding number of the
path the particle takes over its lifetime in this circular universe of the
quantum rotor. It is therefore an observable that holds global information
even though only local terms contribute. For example, the topological
susceptibility \eqref{eq:topsusc} can be used to extract the energy gap of
the quantum rotor. Further discussion on this observable is deferred to
Section~\ref{sec:topsus}.

\subsection{Quantum compact abelian gauge theory} \label{sec:QED}

The essential building blocks of models in high energy physics are
so-called gauge theories. They describe the physics of the force mediating
particles between matter fields. The lattice action of a gauge theory is
constructed from {\em gauge fields} and assumes a particular form based on
the plaquette discussed below. The gauge fields themselves are represented
by group valued variables which are taken from the abelian group $\U(1)$,
or the non-abelian ones, $\U(N)$ or $\SU(N)$ with $N\ge 2$.

Models of high energy physics live in three space and one time dimensions
and take the group $\U(1)$ to describe the electromagnetic, $\SU(2)$ to
describe the weak and $\SU(3)$ to describe the strong interactions, the
latter being the interaction between quarks and gluons. While
investigating these models in $3+1$ dimensions would, of course, be
physically most interesting, they are, unfortunately, presently too
demanding for our approach. However, lower dimensional systems capture
already a number of the essential characteristics of these realistic
models and being able to solve gauge theories in lower dimensions would
open a most promising path to address eventually interesting and important
questions in high energy and condensed matter physics.

As a very first target theory we will consider here 2-dimensional compact
$\U(1)$ lattice gauge theory \cite{Gattringer:2010zz}. The ratio of
interest \eqref{eq:ratio} for this model has in this case the denominator
\begin{align} \label{eq:QED2D}
  \int_{D^{2L^2}}
  \exp\bigg(\beta \sum_{i=0}^{L-1} \sum_{j=0}^{L-1}
  \cos\Big(\phi^a_{i,j} + \phi^b_{i+1,j} - \phi^a_{i,j+1} - \phi^b_{i,j} \Big)
  \bigg)
  \,\rd\bsphi,
\end{align}
and e.g., (using translational invariance) the numerator
\begin{align*}
  &\int_{D^{2L^2}}
  \cos\Big(\phi^a_{0,0} + \phi^b_{1,0} - \phi^a_{0,1} - \phi^b_{0,0} \Big)
  \, \exp\bigg(\beta \sum_{i=0}^{L-1} \sum_{j=0}^{L-1}
  \cos\Big(\phi^a_{i,j} + \phi^b_{i+1,j} - \phi^a_{i,j+1} - \phi^b_{i,j} \Big)
  \bigg)
  \,\rd\bsphi.
\end{align*}
Here and below $D=[-\pi,\pi]$, and we have parametric periodicity where
all indices should be taken modulo $L$.

The above model involves only first order couplings. We can have higher
order couplings by considering the \emph{Wilson loop} with parameters
$r_a$ and $r_b$. The ratio of interest \eqref{eq:ratio} in this case has the
denominator
\begin{align} \label{eq:Wilson}
  &\int_{D^{2L^2}}
  \exp\bigg(\beta \sum_{i=0}^{L-1} \sum_{j=0}^{L-1}
  \cos\Big(
  \phi^a_{i,j} + \phi^a_{i+1,j} + \cdots + \phi^a_{i+r_a,j}
  + \phi^b_{i+r_a,j} + \phi^b_{i+r_a,j+1} + \cdots + \phi^b_{i+r_a,j+r_b} \nonumber\\
  & \qquad\qquad\qquad\qquad\qquad\qquad
  - \phi^a_{i+r_a,j+r_b} - \phi^a_{i+r_a-1,j+r_b} - \cdots - \phi^a_{i,j+r_b} \nonumber\\
  & \qquad\qquad\qquad\qquad\qquad\qquad
  - \phi^b_{i,j+r_b} - \phi^b_{i,j+r_b-1} - \cdots - \phi^b_{i,j}
  \Big)
  \bigg)
  \,\rd\bsphi,
\end{align}
while the numerator should include an extra factor as observable, which
could be the sum
\begin{align*}
  O[\bsphi]
  &\,=\,
  \sum_{i=0}^{L-1} \sum_{j=0}^{L-1}
  \cos\Big(
  \phi^a_{i,j} + \phi^a_{i+1,j} + \cdots + \phi^a_{i+r_a,j}
  + \phi^b_{i+r_a,j} + \phi^b_{i+r_a,j+1} + \cdots + \phi^b_{i+r_a,j+r_b} \\
  & \qquad\quad
   - \phi^a_{i+r_a,j+r_b} - \phi^a_{i+r_a-1,j+r_b} - \cdots - \phi^a_{i,j+r_b}
   - \phi^b_{i,j+r_b} - \phi^b_{i,j+r_b-1} - \cdots - \phi^b_{i,j}
  \Big),
\end{align*}
or just one of the terms in the sum.

The 3-dimensional compact $\U(1)$ lattice gauge theory is more complicated,
involving more than one cosine term in the action. In the case of first
order coupling, the ratio of interest \eqref{eq:ratio} has the denominator
\begin{align} \label{eq:QED3D}
  &\int_{D^{3L^3}}
  \exp\bigg(\beta \sum_{i=0}^{L-1} \sum_{j=0}^{L-1} \sum_{k=0}^{L-1} \bigg[
  \cos\Big(\phi^a_{i,j,k} - \phi^a_{i,j+1,k} - \phi^b_{i,j,k} + \phi^b_{i+1,j,k} \Big) \nonumber\\
  &\qquad\qquad\qquad\qquad\qquad\qquad
  + \cos\Big(\phi^c_{i,j,k} - \phi^c_{i+1,j,k} - \phi^a_{i,j,k} + \phi^a_{i,j,k+1} \Big) \nonumber\\
  &\qquad\qquad\qquad\qquad\qquad\qquad
  + \cos\Big(\phi^b_{i,j,k} - \phi^b_{i,j,k+1} - \phi^c_{i,j,k} + \phi^c_{i,j+1,k} \Big)
  \bigg] \bigg)
  \,\rd\bsphi,
\end{align}
while the numerator should have an extra factor as observable, e.g.,
\[
  O[\bsphi]
  \,=\, \cos\Big(\phi^a_{0,0,0} - \phi^a_{0,1,0} - \phi^b_{0,0,0} + \phi^b_{1,0,0}\Big).
\]

Although using these cosine terms is a very common way of expressing a
$\U(1)$ gauge theory like QED (quantum electrodynamics), they are not
particularly insightful with respect to the $\U(1)$ structure and its
generalizations. The angles $\phi_{i,j,\ldots}^\alpha$ are the angles of
$\U(1)$ elements $\Phi_{i,j,\ldots}^\alpha=e^{\ri\phi_{i,j,\ldots}^\alpha}$
which itself describe the transformation that moves the field from one
lattice point to a neighboring point. The superscript $a$ then denotes
movement increasing the first index $i$ by one, i.e., $\Phi_{i,j}^a$
describes the transformation of the field moving from the lattice point
$(i,j)$ to $(i+1,j)$, while $(\Phi_{i,j}^a)^{-1} = e^{-\ri\phi_{i,j}^a}$
describes its inverse. Similarly $\Phi_{i,j}^b$ describes movement from
$(i,j)$ to $(i,j+1)$. In three dimensions we also have a $c$ direction
describing movement from $(i,j,k)$ to $(i,j,k+1)$. Following this
procedure, $\U(1)$ gauge theories in arbitrarily many dimensions can be
constructed.

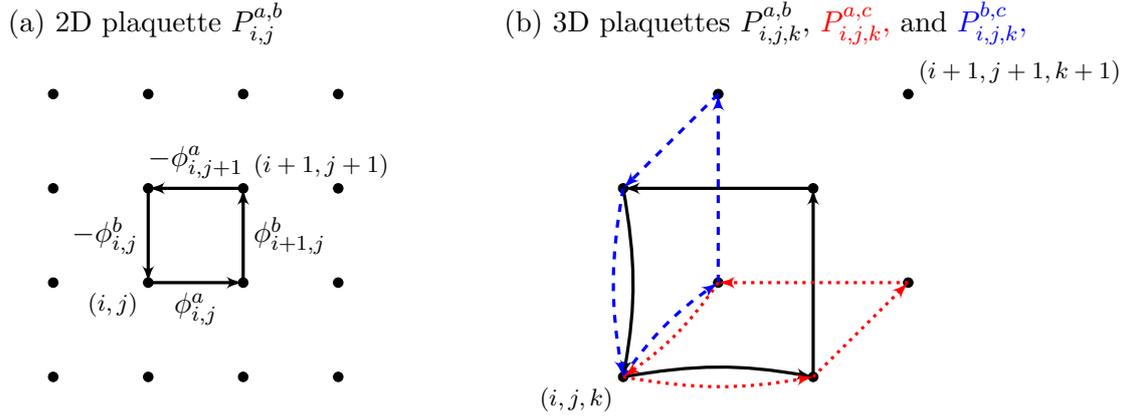
\begin{figure}
  \centering
  \begin{tikzpicture}[scale=1.25,>=latex']
    \node at (-.025,3.75) {(a) 2D plaquette $P_{i,j}^{a,b}$};
    \foreach \i in {-1,...,2}{
      \foreach \j in {0,...,3}{
        \filldraw[black] (\i,\j) circle (.05);
      }
    }
    \draw[->,black,very thick] (0,1) -- (1,1);
    \draw[->,black,very thick] (1,1) -- (1,2);
    \draw[->,black,very thick] (1,2) -- (0,2);
    \draw[->,black,very thick] (0,2) -- (0,1);
    \node[black,below] at (0.5,1) {$\phi_{i,j}^a$};
    \node[black,above] at (0.5,2) {$-\phi_{i,j+1}^a$};
    \node[black,left] at (0,1.5) {$-\phi_{i,j}^b$};
    \node[black,right] at (1,1.5) {$\phi_{i+1,j}^b$};
    \node[black,below left] at (0,1) {\footnotesize $(i,j)$};
    \node[black,above right] at (1,2) {\footnotesize $(i+1,j+1)$};
    \node at (6.5,3.75) {(b) 3D plaquettes {\color{black} $P_{i,j,k}^{a,b}$, {\color{red} $P_{i,j,k}^{a,c}$,} and {\color{blue} $P_{i,j,k}^{b,c}$,}}};
    \filldraw[black] (5,0) circle (.05);
    \filldraw[black] (7,0) circle (.05);
    \filldraw[black] (5,2) circle (.05);
    \filldraw[black] (7,2) circle (.05);
    \filldraw[black] (6,1) circle (.05);
    \filldraw[black] (8,1) circle (.05);
    \filldraw[black] (6,3) circle (.05);
    \filldraw[black] (8,3) circle (.05);
    \node[black,below left] at (5,0) {\footnotesize $(i,j,k)$};
    \node[black,above right] at (8,3) {\footnotesize $(i+1,j+1,k+1)$};
    \draw[->,black,very thick] (5,0) to[out=10,in=170] (7,0);
    \draw[->,black,very thick] (7,0) -- (7,2);
    \draw[->,black,very thick] (7,2) -- (5,2);
    \draw[->,black,very thick] (5,2) to[out=280,in=80] (5,0);
    \draw[->,dashed,blue,very thick] (5,0) to[out=55,in=215] (6,1);
    \draw[->,dashed,blue,very thick] (6,1) -- (6,3);
    \draw[->,dashed,blue,very thick] (6,3) -- (5,2);
    \draw[->,dashed,blue,very thick] (5,2) to[out=260,in=100] (5,0);
    \draw[->,dotted,red,very thick] (5,0) to[out=-10,in=190] (7,0);
    \draw[->,dotted,red,very thick] (7,0) -- (8,1);
    \draw[->,dotted,red,very thick] (8,1) -- (6,1);
    \draw[->,dotted,red,very thick] (6,1) to[out=235,in=35] (5,0);
  \end{tikzpicture}
\caption{Construction of the plaquette in two (left) and three
(right) dimensions. In two dimensions, the plaquette $P_{i,j}^{a,b}$ at
lattice point $(i,j)$ is constructed by transporting the field along the
path $(i,j)\to(i+1,j)\to(i+1,j+1)\to(i,j+1)\to(i,j)$. The field
transformation along this path is given by \eqref{eq:plaquette-path}. The
angles $\phi$ contribute with a positive sign if moving right or up, and
with a negative sign if moving left or down. Left/right movement is
labeled with a superscript $a$ and up/down movement with superscript $b$.
In three dimensions an additional foward/backward movement is denoted by
the superscript $c$. Furthermore, there are three independent smallest
loops that can be taken; namely in the $(a,b)$-plane (black solid line),
$(a,c)$-plane (red dotted line), and $(b,c)$-plane (blue dashed line).
Each of these planes equally contribute a ``2-dimensional'' plaquette. We
note that all links are straight lines which we have bent simply for
better visualization.}\label{fig:QED-plaquette}
\end{figure}

The cosine terms above now follow from the self-interaction of the quantum
field. The (non-trivial) first order discretization of this
self-interaction is therefore given by the smallest non-trivial path
through the physical lattice. In 2-dimensions starting from the lattice
point $(i,j)$ that is
  \begin{align}\label{eq:plaquette-path}
    (i,j)\xrightarrow{\Phi_{i,j}^a}(i+1,j)\xrightarrow{\Phi_{i+1,j}^b}(i+1,j+1)\xrightarrow{(\Phi_{i,j+1}^a)^{-1}}(i,j+1)\xrightarrow{(\Phi_{i,j}^b)^{-1}} (i,j).
  \end{align}
  The self-interaction at the point $(i,j)$ is thus given by the plaquette
  \begin{align}\label{eq:plaquette-definition}
    P_{i,j}^{a,b} = \Phi_{i,j}^a\Phi_{i+1,j}^b(\Phi_{i,j+1}^a)^{-1}(\Phi_{i,j}^b)^{-1} = e^{\ri\left(\phi_{i,j}^a+\phi_{i+1,j}^b-\phi_{i,j+1}^a-\phi_{i,j}^b\right)}
  \end{align}
and the cosines in the 2-dimensional action are precisely given by the
real parts of the plaquettes $P_{i,j}^{a,b}$. Real parts of larger loops
are the Wilson loops. Furthermore, in three or more dimensions, plaquettes
can be built in the $(a,b)$-plane, $(a,c)$-plane, $(b,c)$-plane, etc. The
3-dimensional action is thus given by summing over all three plaquettes at
each lattice point. The construction of the plaquettes in two and three
dimensions is visualized in Figure~\ref{fig:QED-plaquette}.

In a similar way, taking $\Phi_{i,j,\ldots}^\alpha$ from $\U(N)$ or $\SU(N)$
and taking the trace of \eqref{eq:plaquette-definition} would allow for
the construction of other gauge theories such as the weak interaction,
using $\SU(2)$, or QCD (quantum chromodynamics or strong nuclear force),
using $\SU(3)$.

In addition to the self-interaction of the quantum field, the complete QED
action also contains an interaction term with fermions and a topological
term. The fermions are not included in the description shown above because
they are modeled using Grassmann variables which can be integrated
analytically. The analytic expression for these integrals can then be
treated as part of an observable.

The topological term, on the other hand, is very interesting and
notorious. While the self-interaction term is given by the sum over all
plaquette real parts, summing all plaquette imaginary parts yields
$2\pi\,Q[\Phi]$ where $Q[\Phi]$ is by definition the topological charge.
Note that this is the field theoretical analogue of \eqref{eq:topcharge}.
The topological charge contributes to the action as an additional summand
$\ri\theta\, Q[\Phi]$. The coefficient $\theta$ is called the vacuum angle
and is set to $\theta=0$ in the action we will discuss here. Physically
the value of $\theta$ distinguishes between superselection sectors. It is
therefore physically important to have numerical methods that can handle
non-zero values of $\theta$.

Non-zero values of $\theta$ are very prohibitive for many state-of-the-art
methods in computational physics because they render the action complex.
Hence, techniques like Markov chain Monte Carlo would require sampling
from complex ``probability'' distributions. Alternatively, the term could
be treated as a complex-valued observable which leads into the notorious
sign-problem \cite{GHJLV16a,GHJLV16b,Hartung:2020uuj,Troyer:2004ge}.
Although we are not explicitly treating this term here, it is important to
note that this topological term is nothing but a sine of angle differences
and thus shares all periodicity and differentiability properties of the
class of numerical problems we are studying in this work. Hence, the
methods developed here have the potential to address long-standing open
problems in high energy physics, going beyond the capabilities of
state-of-the-art methods.

\section{Recursive strategy for first order couplings} \label{sec:first}

In this section we review and extend the approach from \cite{AGHJLV16}.

The integral of interest takes the generic form \eqref{eq:int1} (not
restricted to the physics integrals) and can be written equivalently as
follows: with $\bsx = (x_0,\ldots,x_{L-1})$ and $x_L\equiv x_0$,
\ifdefined\journalstyle
    \begin{align} \label{eq:int2}
     \calI
     &\,=\, \int_D\cdots\int_D
     f_0\big(x_0,x_1\big)
     f_1\big(x_1,x_2\big)f_2\big(x_2,x_3\big)
     \cdots
     f_{L-1}\big(x_{L-1},x_0\big) \,\rd x_0 \cdots \rd x_{L-1} \nonumber\\
     &\,=\, \int_D \bigg[\int_D\cdots
     \left(\int_D
     \left(\int_D f_0\big(x_0,x_1\big) f_1\big(x_1,x_2\big)\,\rd x_1\right)
     f_2\big(x_2,x_3\big)
     \,\rd x_2 \right)
     \cdots f_{L-1}\big(x_{L-1},x_0\big) \,\rd x_{L-1} \bigg] \,\rd x_0.
    \end{align}
\else
    \begin{align} \label{eq:int2}
     \calI
     &\,=\, \int_D\cdots\int_D
     f_0\big(x_0,x_1\big)
     f_1\big(x_1,x_2\big)f_2\big(x_2,x_3\big)
     \cdots
     f_{L-1}\big(x_{L-1},x_0\big) \,\rd x_0 \cdots \rd x_{L-1} \nonumber\\
     &\,=\, \int_D \bigg[\int_D\cdots
     \left(\int_D
     \left(\int_D f_0\big(x_0,x_1\big) f_1\big(x_1,x_2\big)\,\rd x_1\right)
     f_2\big(x_2,x_3\big)
     \,\rd x_2 \right) \nonumber\\
     &\qquad\qquad\qquad\qquad\qquad\qquad\qquad\qquad\qquad\qquad\qquad
     \cdots f_{L-1}\big(x_{L-1},x_0\big) \,\rd x_{L-1} \bigg] \,\rd x_0.
    \end{align}
\fi
We say that this problem has \emph{first order couplings} because each
function $f_i$ depends on a variable $x_i$ and its next neighbor
$x_{i+1}$, with the assumed \emph{parametric periodicity} that $x_L\equiv
x_0$. The way we have grouped the integrals in \eqref{eq:int2} would imply
conceptually that we first integrate over $x_1$, and then $x_2$ and so on
until $x_{L-1}$, and finish with $x_0$. But the underlying parametric
periodicity means that there is no absolute ordering of variables. We
could start from any variable and go either forward or backward through
the sequence. Our choice to begin from $x_1$ is just for notational
convenience. The reason to leave $x_0$ for the last is because it appears
in both $f_0$ and $f_{L-1}$, far apart in our expression, and so a bit
clumsy if we were to tackle it first.

For physics integrals involving an observable function which has up to
first order couplings, we will end up with an extra factor for the product
function. This extra factor can be grouped with an existing factor
involving the same variables. Due to the parametric periodicity, without
loss of generality, we can assume that this observable function has been
grouped with the factor~$f_0$.

Following the recursive integration strategy described in \cite{AGHJLV16},
each one-dimensional integral in \eqref{eq:int2} is approximated by the
same one-dimensional quadrature rule with points $t_0,\ldots,t_{n-1}\in D$
and weights $w_0,\ldots,w_{n-1}\in \bbR$, to arrive at the approximation
\ifdefined\journalstyle
    \begin{align} \label{eq:prod2}
     \calQ
     &\,=\, \sum_{k_0=0}^{n-1} w_{k_0} \bigg[
     \sum_{k_{L-1}=0}^{n-1} w_{k_{L-1}} \cdots \bigg(\sum_{k_2=0}^{n-1} w_{k_2}
     \bigg(\sum_{k_1=0}^{n-1} w_{k_1} f_0\big(t_{k_0},t_{k_1}\big) f_1\big(t_{k_1},t_{k_2}\big)\bigg)
     f_2\big(t_{k_2},t_{k_3}\big)\bigg)
     \cdots
     f_{L-1}\big(t_{k_{L-1}},t_{k_0}\big) \bigg].
    \end{align}
\else
    \begin{align} \label{eq:prod2}
     \calQ
     &\,=\, \sum_{k_0=0}^{n-1} w_{k_0} \bigg[
     \sum_{k_{L-1}=0}^{n-1} w_{k_{L-1}} \cdots \bigg(\sum_{k_2=0}^{n-1} w_{k_2}
     \bigg(\sum_{k_1=0}^{n-1} w_{k_1} f_0\big(t_{k_0},t_{k_1}\big) f_1\big(t_{k_1},t_{k_2}\big)\bigg)
     f_2\big(t_{k_2},t_{k_3}\big)\bigg) \nonumber\\
     &\qquad\qquad\qquad\qquad\qquad\qquad\qquad\qquad\qquad\qquad\qquad\qquad\qquad\qquad
     \cdots
     f_{L-1}\big(t_{k_{L-1}},t_{k_0}\big) \bigg].
    \end{align}
\fi
An important fact which was perhaps not sufficiently emphasized in
\cite{AGHJLV16} is that this approximation \eqref{eq:prod2} is precisely
the tensor product rule \eqref{eq:prod1}. Therefore, the error of this
approximation is $\calO(n^{-\alpha})$ if the one-dimensional rule has
error $\calO(n^{-\alpha})$. Furthermore, if each $f_i$ is smooth and
periodic, then the rectangle rule (i.e., equally spaced points $t_k$ on
$D$ and equal weights $w_k = 1/n$) becomes the trapezoidal rule, and its
error converges in $\calO(n^{-\alpha})$ by the Euler--Maclauren formula,
where $\alpha$ is the smoothness order of the function (or even
exponentially fast in $n$ for analytic functions).

The expression~\eqref{eq:prod2} is equivalent to a time integrator in
the quantum rotor model (and other linear Schr\"odinger operators). This
is mainly due to the relatively simple connection between the Lagrangian
and Hamiltonian formulations of classical quantum mechanics after
discretization. These similarities disappear for non-trivial quantum field
theories, to the point at which the Hamiltonian becomes so complex, e.g.\
in lattice QCD, that it cannot serve as guiding principle any more.

Assuming that the cost of one evaluation of $f_i$ is $\calO(1)$, a naive
implementation of this method by direct calculation of \eqref{eq:prod2}
would have cost $\calO(n^L)$ and so suffers from the curse of
dimensionality. This is summarized under Scenario (A0) in
Table~\ref{tab1}. In the following, we first describe the efficient
strategy from \cite{AGHJLV16} which improves the cost to Scenarios
(A1)--(A4) in Table~\ref{tab1}. We then extend our discussion to even more
favorable Scenarios (A5)--(A7) to complete the table. Finally we explain
that Table~\ref{tab1} also holds when the one-dimensional domain $D$ is
replaced by an $s$-dimensional domain.

In Table~\ref{tab1} (and also later in Table~\ref{tab2}) we assume that
the chosen {\tt eig} procedure returns eigenvalues and eigenvectors to the
desired working precision and its cost can be expressed only in terms of
$n$. In a similar simplification, we assume also that the quadrature
weights in the matrix $W$ can be obtained up to working precision with
negligible cost. These assumptions will hold then also for the rest of
this paper.

\begin{table} [t]
\caption{\textbf{Cost of recursive strategy for first order couplings.}
$M_i$ is the $n\times n$ matrix of $f_i$ at quadrature points. $W$ is an
$n\times n$ diagonal matrix with quadrature weights on the diagonal. {\tt
eig} returns a diagonal matrix of eigenvalues. {\tt fft} takes the first
column of a circulant matrix and returns a diagonal matrix of eigenvalues.
In all cases the quadrature error is $\calO(n^{-\alpha})$, with $\alpha$
determined by the quadrature rule. The strategy extends to an $L$-fold
product of $s$-dimensional integrals with the quadrature rule replaced by
an $s$-dimensional cubature rule. The cost is independent of~$s$, while
the error is $\calO(n^{-\alpha})$, with $\alpha$ determined by the
cubature rule and with implied constant dependent on $s$.}
 \label{tab1}
 \begin{center}
 \setlength{\extrarowheight}{1pt}
 \begin{tabular}{|l|l|c|}
 \hline
 Scenario & Strategy & Cost \\
 \hline
 \hline
 (A0) naive implementation &
 $\begin{array}{l}
 \calQ = \mbox{direct product calculation}
 \end{array}$
 & $n^L$ \\
 \hline
 (A1) recursive integration &
 $\begin{array}{l}
  B = W^{1/2} M_0 W M_1 W \cdots M_{L-1} W^{1/2} \\
  \calQ = \sum_{k=0}^{n-1} B_{k,k}
 \end{array}$
 & $L\,n^{3}$ \\
 & & \vspace{-0.4cm} \\
 \hline
 (A2) $M_i = M$ &
 $\begin{array}{l}
  A = W^{1/2} M W^{1/2} \\
  B = A^L \\
  \calQ = \sum_{k=0}^{n-1} B_{k,k}
 \end{array}
 $
 & $\log(L)\,n^{3}$ \\
 & & \vspace{-0.4cm} \\
 \hline
 (A3) $M_i=M$ diagonalizable &
 $\begin{array}{l}
  A = W^{1/2} M W^{1/2} \\
  \Lambda = {\tt eig}(A) \\
  \calQ = \sum_{k=0}^{n-1} \Lambda_{k,k}^L
 \end{array}
 $
 & $n^{3}$ \\
 & & \vspace{-0.4cm} \\
 \hline
 (A4) $M_i = M$ except $M_0$ &
 $\begin{array}{l}
  A = W^{1/2} M W^{1/2} \\
  B = W^{1/2} M_0 W^{1/2} A^{L-1} \\
  \calQ = \sum_{k=0}^{n-1} B_{k,k}
 \end{array}
 $
 & $\log(L)\,n^{3}$ \\
 & & \vspace{-0.4cm} \\
 \hline
 (A5) $M_i$ circulant &
 $\begin{array}{l}
  \Lambda_i = {\tt fft}(M_i/n) \mbox{ for each $i$} \\
  \calQ = \sum_{k=0}^{n-1} \prod_{i=0}^{L-1} (\Lambda_i)_{k,k}
 \end{array}
 $
 & $L\,n\log(n)$ \\
 & & \vspace{-0.4cm} \\
 \hline
 (A6) $M_i = M$ circulant &
 $\begin{array}{l}
  \Lambda = {\tt fft}(M/n) \\
  \calQ = \sum_{k=0}^{n-1} \Lambda_{k,k}^L
 \end{array}
 $
 & $n\log(n)$ \\
 & & \vspace{-0.4cm} \\
 \hline
 \begin{tabular}{l}
 \!\!\!\!(A7) $M_i = M$ except $M_0$ \\
 \quad\;\, all circulant
 \end{tabular}
 &
 $\begin{array}{l}
  \Lambda_0 = {\tt fft}(M_0/n) \\
  \Lambda = {\tt fft}(M/n) \\
  \calQ = \sum_{k=0}^{n-1} (\Lambda_0)_{k,k}\, \Lambda_{k,k}^{L-1}
 \end{array}
 $
 & $n\log(n)$ \\
 \hline
\end{tabular}
\end{center}
\end{table}

\subsection{Recursive numerical integration} \label{sec:first-recur}

Let $M_i$ denote the $n\times n$ matrix with entries
\begin{align} \label{eq:Mi}
  (M_i)_{p,q} \,=\, f_i(t_p,t_q) \qquad\mbox{for}\quad p,q=0,\ldots,{n-1},
\end{align}
and let $W$ denote the $n\times n$ diagonal matrix with the weights
$w_0,\ldots,w_{n-1}$ on the diagonal. Then we can express the innermost
sum in \eqref{eq:prod2} as
\begin{align*}
  \sum_{k_1=0}^{n-1} w_{k_1} f_0\big(t_{k_0},t_{k_1}\big) f_1\big(t_{k_1},t_{k_2}\big)
  &\,=\, \sum_{k_1=0}^{n-1} (M_0)_{k_0,k_1}w_{k_1}^{1/2} \, w_{k_1}^{1/2} (M_1)_{k_1,k_2} \\
  &\,=\, \sum_{k_1=0}^{n-1} (M_0 W^{1/2})_{k_0,k_1}\, (W^{1/2} M_1)_{k_1,k_2}
  \,=\, (M_0 W M_1)_{k_0,k_2},
\end{align*}
where we used the properties that pre-multiplying by a diagonal matrix
scales the rows while post-multiplying by a diagonal matrix scales the
columns. In turn, we have
\begin{align*}
 \sum_{k_2=0}^{n-1} w_{k_2}
 \bigg(\sum_{k_1=0}^{n-1} w_{k_1} f_0\big(t_{k_0},t_{k_1}\big) f_1\big(t_{k_1},t_{k_2}\big)\bigg)
 f_2\big(t_{k_2},t_{k_3}\big)
  &\,=\,
  \sum_{k_2=0}^{n-1} w_{k_2} (M_0 W M_1)_{k_0,k_2} (M_2)_{k_2,k_3} \\
  &\,=\, (M_0 W M_1 W M_2)_{k_0,k_3}.
\end{align*}
This eventually leads to
\begin{align} \label{eq:fast1}
 \calQ
 &\,=\, \sum_{k_0=0}^{n-1} w_{k_0} (M_0 W M_1 W M_2 W \cdots M_{L-1})_{k_0,k_0} \nonumber\\
 &\,=\, \sum_{k_0=0}^{n-1} (W^{1/2} M_0 W M_1 W M_2 W \cdots M_{L-1} W^{1/2})_{k_0,k_0} \nonumber\\
 &\,=\, {\rm trace}(B),
 \qquad\mbox{with}\qquad B \,=\, W^{1/2} M_0 W M_1 W M_2 W \cdots M_{L-1} W^{1/2}.
\end{align}

Hence, it suffices to compute the matrix $B$ by successive matrix
multiplications, and then summing up the diagonal entries of $B$. Assuming
that the cost for multiplying two $n\times n$ dense matrices is
$\calO(n^\mu)$ with $\mu\le 3$, the cost of the recursive strategy is
$\calO(L\,n^\mu)$. This is summarized as Scenario (A1) in
Table~\ref{tab1}.

For physics integrals involving an observable function which has up to
first order couplings, as we explained before this can be grouped with the
factor $f_0$. So this situation is also covered by Scenario (A1).

For numerical stability, it may be necessary to scale the intermediate
matrix multiplications when implementing this method. We recommend
scaling the matrix so that a prescribed norm is $1$.

\subsection{Identical matrices} \label{sec:first-same}

In the special case where all the functions $f_i$ are equal so that all
matrices are identical, $M_i = M$, we have
\[
 B \,=\, A^L, \qquad\mbox{with}\qquad A \,=\, W^{1/2} M W^{1/2}.
\]
It suffices to compute the $L$th power of $A$ (using e.g., the method of
``exponentiation by squaring'') and then sum up the diagonal entries of
the resulting matrix. The cost is then $\calO(\log(L)\,n^\mu)$. This is
summarized as Scenario (A2) in Table~\ref{tab1}.

If $A$ is diagonalizable, that is,
\[
  A \,=\, P\,\Lambda\, P^\top,
\]
with $P$ an orthogonal matrix and with the eigenvalues of $A$ on the
diagonal of $\Lambda$, then since $B = A^L = P\,\Lambda^L\, P^\top$ we
conclude that
\begin{equation*} 
  \calQ \,=\, {\rm trace}(A^L) \,=\, {\rm trace}(\Lambda^L) \,=\, \sum_{k=0}^{n-1} \Lambda_{k,k}^L.
\end{equation*}
Thus we just need to find the eigenvalues of $A$, raise each of them to
the $L$th power, and then sum them up. The cost is then $\calO(n^\xi)$,
which is dominated by the cost for the eigenvalue decomposition, generally
with $\xi\le 3$. This is summarized as Scenario (A3) in Table~\ref{tab1}.
Such a scenario can occur when the functions $f_i$ are symmetric, i.e.,
$f_i(u,v) = f_i(v,u)$. It is then interesting to see whether strategy (A2)
or (A3) is more efficient in practice.

For an integral with an observable function that has been grouped with
$f_0$, if all other functions $f_i$ are equal, then we arrive at Scenario
(A4) in Table~\ref{tab1}, which effectively has the same cost as Scenario
(A2).

\subsection{Cost saving by FFT} \label{sec:first-fft}

We now extend the strategy beyond \cite{AGHJLV16}. If
\begin{enumerate}
\item each function $f_i$ depends only on the difference of the two
    arguments, i.e., $f_i(u,v) = \kappa_i(v-u)$ for some function
    $\kappa_i$, and
\item each function $\kappa_i$ is \emph{periodic}, and
\item we have equally spaced points with equal weights $1/n$ (i.e., we
    have the rectangle rule),
\end{enumerate}
then the matrix $M_i$ is \emph{circulant} so that FFT can be used to find
the eigenvalues.

While the restriction $f_i(u,v)=\kappa_i(v-u)$ may seem very
restrictive at first glance, it should be noted that $\sigma$-models and
gauge theories satisfy this structure.

If the functions $f_i$ are different, then analogously to \eqref{eq:fast1}
we write (leaving aside the weights to be taken into account at the end)
\[
  B \,=\, M_0 M_1\cdots M_{L-1},
  \qquad\mbox{with}\qquad M_i \,=\, \Fourier\, \Lambda_i\, \Fourier^{\dagger},
\]
where $\Fourier$ is the unitary Fourier matrix and $\dagger$ is the
Hermitian adjoint. Since the Fourier matrix is unitary, we have
\[
 B \,=\, \Fourier\, \Lambda_0\, \Lambda_1 \cdots \Lambda_{L-1}\, \Fourier^{\dagger},
\]
and thus
\[
  \calQ \,=\, n^{-L}\, {\rm trace}(B)
  \,=\, n^{-L} \sum_{k=0}^{n-1} (\Lambda_0)_{k,k} (\Lambda_1)_{k,k} \cdots (\Lambda_{L-1})_{k,k}.
\]
So we carry out FFT on the first column of each of the matrices
$M_0,M_1,\ldots,M_{L-1}$ to find their eigenvalues, multiply the resulting
diagonal matrices elementwise, and then sum up the resulting diagonal and
divide by $n^L$. The cost is therefore $\calO(L\,n\log(n))$. This is
summarized as Scenario (A5) in Table~\ref{tab1}.

If all the functions $f_i$ are equal, then we carry out FFT only once to
find the eigenvalues of the common matrix $M$, raise each eigenvalue to
the power $L$, and then sum up the results. The cost is then reduced to
$\calO(n\log(n))$. This is summarized as Scenario (A6) in
Table~\ref{tab1}.

If an observable function is present as explained, then $f_0$ is different
while the rest of the $f_i$ are equal. In this case, with $M_0 =
\Fourier\,\Lambda_0\, \Fourier^{\dagger}$ and $M = \Fourier\,\Lambda\,
\Fourier^{\dagger}$, we have
\[
 B \,=\, \Fourier\, \Lambda_0\, \Lambda^{L-1}\, \Fourier^{\dagger},
\]
and thus
\[
  \calQ \,=\, n^{-L}\, {\rm trace}(B) \,=\, n^{-L} \sum_{k=0}^{n-1} (\Lambda_0)_{k,k}\,\Lambda_{k,k}^{L-1}.
\]
The cost is again $\calO(n\log(n))$, and this is summarized as Scenario
(A7) in Table~\ref{tab1}. In all cases it is numerically better to perform
the scaling by weights in each step, see Table~\ref{tab1}.
Additionally, we recommend scaling the columns of the circulant
matrices to have a prescribed vector norm of $1$, see the constant
\texttt{c} in the Julia code in Section~\ref{sec:Julia}.

\subsection{Extension to the $L$-fold product of $s$-dimensional integrals}
\label{sec:first-ext}

We conclude this section by noting that the recursive strategy extends
easily to the situation where the domain $D$ in \eqref{eq:int2} is
replaced by an $s$-dimensional domain $D^s$ as in \eqref{eq:intLsb}, or
equivalently,
\begin{align} \label{eq:intLs}
 \calI
 &\,=\, \int_{D^s} \cdots \int_{D^s}
 \prod_{i=0}^{L-1} f_i\big(\bsx_i,\bsx_{i+1}\big) \,\rd\bsx_0\cdots\rd\bsx_{L-1},
\end{align}
where $\bsx_i = (x_{i,0},\ldots,x_{i,s-1})\in D^s$, with
\[
  x_{i,j} \equiv x_{i\bmod L,\; j\bmod s} \qquad \mbox{for all}\quad i,j\in\bbN.
\]

In this case the one-dimensional quadrature rule in \eqref{eq:prod2}
becomes an $s$-dimensional cubature rule with points
$\bst_0,\ldots,\bst_{n-1} \in D^s$ and weights
$w_0,\ldots,w_{n-1}\in\bbR$, and the matrices $M_i$ in \eqref{eq:Mi}
become
\begin{align*}
  (M_i)_{p,q} \,=\, f_i(\bst_p,\bst_q) \qquad\mbox{for}\quad p,q=0,\ldots,{n-1}.
\end{align*}
Thus we have again Scenarios~(A0)--(A4) as in Table~\ref{tab1}.

Sufficient conditions to arrive at circulant matrices for the more
favorable Scenarios~(A5)--(A7) are as follows:
\begin{enumerate}
\item each function $f_i$ depends only on the difference of the two
    arguments, i.e., $f_i(\bsu,\bsv) = \kappa_i(\bsv-\bsu)$ for some
    function $\kappa_i : D^s \to \bbR$, and
\item each function $\kappa_i$ is \emph{periodic with respect to each
    of the $s$ components}, and
\item we have a \emph{lattice cubature rule} $\bst_k = (k\bsz\bmod
    n)/n$ for each $k=0,\ldots,n-1$, see \eqref{eq:lat}.
\end{enumerate}
The group structure of lattice cubature rule means that the difference of
two lattice points is another lattice point. Combining this with equal
weights $1/n$ gives us circulant matrices $M_i$. This is the main
motivation here for favoring lattice cubature rules above all other
cubature rules!

We stress that the cost in all scenarios is \emph{independent of $s$}.
However, the error is $\calO(n^{-\alpha})$, where $\alpha$ is determined
by the cubature rule and the implied constant may depend on~$s$.

\section{Recursive strategy for higher order couplings} \label{sec:higher}

Consider now an integrand which is a product of factors involving
\emph{higher order couplings of order $r$}, with $1\le r\le L-1$, with the
underlying parametric periodicity that $x_i \equiv x_{i\bmod L}$ for all
$i\in\bbN$,
\begin{align} \label{eq:intr}
 \calI
 &\,=\, \int_{D^L} \prod_{i=0}^{L-1} f_i\big(x_i,x_{i+1},\ldots,x_{i+r}\big)
 \,\rd \bsx \\
 &\,=\, \int_D\cdots\int_D
 f_0\big(x_0,x_1,\ldots,x_r\big)f_1\big(x_1,x_2,\ldots,x_{r+1}\big)
 \cdots f_{r}\big(x_{r},x_{r+1},\ldots,x_{2r}\big) \nonumber\\
 &\qquad\qquad\qquad\qquad\qquad\qquad\qquad\qquad\qquad\qquad
 \cdots
 f_{L-1}\big(x_{L-1},x_0,x_1,\ldots,x_{r-1}\big) \,\rd x_0 \cdots \rd x_{L-1}. \nonumber
\end{align}
In this section we generalize the recursive strategies from
Section~\ref{sec:first} using \emph{a tensor product of $r$-dimensional
cubature rules}. Our approach is essentially to turn the given integral
into an $(L/r)$-fold product of $r$-dimensional integrals, a formulation
that we discussed in Section~\ref{sec:first-ext} (replacing $L$ by $L/r$
and $s$ by $r$). To avoid confusion, we summarize our findings in
Table~\ref{tab2} for easy comparison.

For simplicity of presentation we will explain this by considering the
special case $L=14$ and $r=3$:
\begin{align} \label{eq:I-14-3}
 \calI
 &\,=\, \int_D\cdots\int_D
 f_0\big(x_0,x_1,x_2,x_3\big)\, f_1\big(x_1,x_2,x_3,x_4\big)\, f_2\big(x_2,x_3,x_4,x_5\big) \nonumber\\
 &\qquad\qquad\qquad
 f_3\big(x_3,x_4,x_5,x_6\big)\, f_4\big(x_4,x_5,x_6,x_7\big)\, f_5\big(x_5,x_6,x_7,x_8\big) \nonumber\\
 &\qquad\qquad\qquad
 f_6\big(x_6,x_7,x_8,x_9\big)\, f_7\big(x_7,x_8,x_9,x_{10}\big)\,f_8\big(x_8,x_9,x_{10},x_{11}\big) \nonumber\\
 &\qquad\qquad\qquad
 f_9\big(x_9,x_{10},x_{11},x_{12}\big)\,f_{10}\big(x_{10},x_{11},x_{12},x_{13}\big)\, f_{11}\big(x_{11},x_{12},x_{13},x_0\big) \nonumber\\
 &\qquad\qquad\qquad\qquad\qquad\qquad
 f_{12}\big(x_{12},x_{13},x_0,x_1\big)\, f_{13}\big(x_{13},x_0,x_1,x_2\big)
 \,\rd x_0 \cdots \rd x_{13}.
\end{align}
We have deliberately chosen a value of $L$ that is not a multiple of $r$.

We group every three ($=r$) consecutive variables together as follows:
\begin{align*}
 \calI
 &\,=\, \int_{D^3} \int_{D^3} \int_{D^3} \int_{D^3}
 \theta_0\big((x_0,x_1,x_2),(x_3,x_4,x_5)\big)\\
 &\qquad\qquad\qquad\qquad\quad
 \theta_1\big((x_3,x_4,x_5),(x_6,x_7,x_8)\big)\\
 &\qquad\qquad\qquad\qquad\quad
 \theta_2\big((x_6,x_7,x_8),(x_9,x_{10},x_{11})\big)\\
 &\qquad\qquad\qquad\qquad\quad
 \theta_3\big((x_9,x_{10},x_{11}),(x_0,x_1,x_2)\big) \\
 &\qquad\qquad\qquad\qquad\quad
 \,\rd (x_0,x_1,x_2) \,\rd (x_3,x_4,x_5)\,\rd(x_6,x_7,x_8)\,\rd(x_9,x_{10},x_{11}),
\end{align*}
where we defined
\begin{align*}
 \theta_0\big((x_0,x_1,x_2),(x_3,x_4,x_5)\big)
 &\,:=\, f_0\big(x_0,x_1,x_2,x_3\big)\, f_1\big(x_1,x_2,x_3,x_4\big)\, f_2\big(x_2,x_3,x_4,x_5\big)
 \\
 \theta_1\big((x_3,x_4,x_5),(x_6,x_7,x_8)\big)
 &\,:=\,
 f_3\big(x_3,x_4,x_5,x_6\big)\, f_4\big(x_4,x_5,x_6,x_7\big)\, f_5\big(x_5,x_6,x_7,x_8\big)
 \\
 \theta_2\big((x_6,x_7,x_8),(x_9,x_{10},x_{11})\big)
 &\,:=\,
 f_6\big(x_6,x_7,x_8,x_9\big)\, f_7\big(x_7,x_8,x_9,x_{10}\big)\,f_8\big(x_8,x_9,x_{10},x_{11}\big),
\end{align*}
with the exceptional last one
\begin{multline*}
  \theta_3((x_9,x_{10},x_{11}),(x_0,x_1,x_2))
  \,:=\,
  \int_D\int_D
  f_9\big(x_9,x_{10},x_{11},x_{12}\big)\,f_{10}\big(x_{10},x_{11},x_{12},x_{13}\big)
  \\
 f_{11}\big(x_{11},x_{12},x_{13},x_0\big)\,
 f_{12}\big(x_{12},x_{13},x_0,x_1\big)\, f_{13}\big(x_{13},x_0,x_1,x_2\big)
  \,\rd x_{12} \,\rd x_{13},
\end{multline*}
which took care of the remaining factors that arise because $L=14$ is not
a multiple of $r=3$.

\begin{table} [t]
\caption{\textbf{Cost of recursive strategy for order $r$ couplings.} $L$
is a multiple of $r$. $M_i$ is the $n\times n$ matrix of $\theta_i$ at
cubature points. $W$ is an $n\times n$ diagonal matrix with cubature
weights on the diagonal. {\tt eig} returns a diagonal matrix of
eigenvalues. {\tt fft} takes the first column of a circulant matrix and
returns a diagonal matrix of eigenvalues. In all cases the error is
$\calO(n^{-\alpha})$, with $\alpha$ determined by the cubature rule and
the implied constant dependent on $r$.}
 \label{tab2}
 \begin{center}
 \setlength{\extrarowheight}{1pt}
 \begin{tabular}{|l|l|c|}
 \hline
 Scenario & Strategy & Cost \\
 \hline\hline
 (B1) recursive integration &
 $\begin{array}{l}
  B = W^{1/2} M_0 W M_1 W M_2 \cdots M_{L/r-1} W^{1/2} \\
  \calQ = \sum_{k=0}^{n-1} B_{k,k}
 \end{array}$
 & $L\,n^3$ \\
 & & \vspace{-0.4cm} \\
 \hline
 (B2) $M_i = M$ &
 $\begin{array}{l}
   A = W^{1/2} M W^{1/2} \\
  B = A^{L/r} \\
  \calQ = \sum_{k=0}^{n-1} B_{k,k}
 \end{array}
 $
 & $\log(L/r)\,n^3$ \\
 & & \vspace{-0.4cm} \\
 \hline
 (B3) $M_i=M$ diagonalizable &
 $\begin{array}{l}
   A = W^{1/2} M W^{1/2} \\
   \Lambda = {\tt eig}(A) \\
  \calQ = \sum_{k=0}^{n-1}  \Lambda_{k,k}^{L/r}
 \end{array}
 $
 & $n^3$\\
 & & \vspace{-0.4cm} \\
 \hline
 (B4) $M_i = M$ except $M_0$ &
 $\begin{array}{l}
  A = W^{1/2} M W^{1/2} \\
  B = W^{1/2} M_0 W^{1/2}\, A^{L/r-1} \\
  \calQ = \sum_{k=0}^{n-1} B_{k,k}
 \end{array}
 $
 & $\log(L/r)\,n^3$ \\
 & & \vspace{-0.4cm} \\
 \hline
 (B5) $M_i$ circulant &
 $\begin{array}{l}
  \Lambda_i = {\tt fft}(M_i/n)  \mbox{ for each $i$} \\
  \calQ = \sum_{k=0}^{n-1} \prod_{i=0}^{L/r-1} (\Lambda_i)_{k,k}
 \end{array}
 $
 & $(L/r)\,n\log(n)$ \\
 & & \vspace{-0.4cm} \\
 \hline
 (B6) $M_i = M$ circulant &
 $\begin{array}{l}
  \Lambda = {\tt fft}(M/n) \\
  \calQ = \sum_{k=0}^{n-1} \Lambda_{k,k}^{L/r}
 \end{array}
 $
 & $n\log(n)$ \\
 & & \vspace{-0.4cm} \\
 \hline
 \begin{tabular}{l}
 \!\!\!\!(B7) $M_i = M$ except $M_0$ \\
 \quad\;\, all circulant
 \end{tabular}
 &
  $\begin{array}{l}
  \Lambda_0 = {\tt fft}(M_0/n) \\
  \Lambda = {\tt fft}(M) \\
  \calQ = \sum_{k=0}^{n-1} (\Lambda_0)_{k,k}\,\Lambda_{k,k}^{L/r-1}
 \end{array}
 $
 & $n\log(n)$ \\
 \hline
\end{tabular}
\end{center}
\end{table}

Next we apply a $3$-dimensional cubature rule with $n$ points
$\bst_0,\ldots,\bst_{n-1}$ and weights $\omega_0,\ldots$, $\omega_{n-1}$
to each integral over $D^3$, to obtain
\begin{align} \label{eq:familiar}
 \calQ
 &\,=\, \sum_{k_0=0}^{n-1} \omega_{k_0} \sum_{k_3=0}^{n-1} \omega_{k_3} \sum_{k_6=0}^{n-1} \omega_{k_6} \sum_{k_9=0}^{n-1} \omega_{k_9}\,
 \theta_0\big(\bst_{k_0},\bst_{k_3}\big)\,
 \theta_1\big(\bst_{k_3},\bst_{k_6}\big)\,
 \theta_2\big(\bst_{k_6},\bst_{k_9}\big)\,
 \widetilde{\theta}_3\big(\bst_{k_9},\bst_{k_0}\big),
\end{align}
with $\theta_3$ approximated by $\widetilde{\theta}_3$, obtained by the
same cubature rule (projected to two dimensions)
\begin{multline*}
  \widetilde{\theta}_3\big((x_9,x_{10},x_{11}),(x_0,x_1,x_2)\big)
  \,:=\,
  \sum_{k=0}^{n-1} \omega_k\,
  f_9\big(x_9,x_{10},x_{11},t_{k,1}\big)\,f_{10}\big(x_{10},x_{11},t_{k,1},t_{k,2}\big)
  \\
 f_{11}\big(x_{11},t_{k,1},x_{13},x_0\big)\,
 f_{12}\big(t_{k,1},t_{k,2},x_0,x_1\big)\, f_{13}\big(t_{k,2},x_0,x_1,x_2\big).
\end{multline*}
Observe that the expression \eqref{eq:familiar} takes the same form as
\eqref{eq:prod2}.

In general, if $L$ is a multiple of $r$, then we rewrite the integral
\eqref{eq:intr} in the form
\begin{align*}
 \calI
 &\,=\, \int_{D^r} \cdots \int_{D^r}
 \prod_{k=0}^{L/r-1} \theta_i\big(\bsy_i,\bsy_{i+1}\big) \,\rd\bsy_0 \cdots \rd\bsy_{L/r},
\end{align*}
where $\bsy_i = (x_{ri},x_{ri+1},\cdots,x_{ri+r-1})$, with $\bsy_i \equiv
\bsy_{i\bmod L/r}$, and
\[
 \theta_i(\bsu) \,:=\, \theta_i(\bsu_{0:r-1},\bsu_{r:2r-1})
 \,:=\, \prod_{k=0}^{r-1} f_{ri+k}(u_k,u_{k+1},\ldots,u_{k+r}).
\]
Thus, with a general $r$-dimensional cubature rule, the matrices of
interest are now
\[
  (M_i)_{p,q} \,=\, \theta_i(\bst_p,\bst_q), \qquad p,q=0,\ldots,n-1,
\]
and as before $W$ denotes the $n\times n$ diagonal matrix with the weights
$w_0,\ldots,w_{n-1}$ on the diagonal. Then similarly to \eqref{eq:fast1}
we obtain
\[
  \calQ \,=\, {\rm trace}(B),
  \qquad B \,=\, W^{1/2} M_0 W M_1 W M_2 \cdots M_{L/r-1} W^{1/2},
\]
which leads to the scenarios in Table~\ref{tab2}, completely analogous to
Table~\ref{tab1}. When $L$ is not a multiple of~$r$, one of the matrices
will need to be adjusted as we have demonstrated in
$\widetilde{\theta}_3$. Without loss of generality, for notational
convenience we can make $M_0$ the adjusted one.

A noteworthy difference between Scenarios~(B5)--(B7) in Table~\ref{tab2}
and Scenarios~(A5)--(A7) in Table~\ref{tab1} is that the matrices $M_i$
are now determined by the functions $\theta_i$ which in turn are formed by
products of the functions $f_i$. Currently we are not aware of sufficient
conditions on $f_i$ that will lead to circulant matrices $M_i$. So it is
possible that Scenarios~(B5)--(B7) are unreachable.

\section{Application to the quantum rotor} \label{sec:app-rotor}

\subsection{1D first order couplings} \label{sec:Julia}

We now apply the recursive strategy to the quantum rotor problem. Both
integrals for the numerator and denominator of our ratio of interest
\eqref{eq:rotor} are of the form
\[
  \int_{D^L} \prod_{i=0}^{L-1} f_i(x_{i+1}-x_i)\,\rd\bsx,
\]
with (after a change of variables) $D = [0,1]$ and
\[
  f_i(x) = f(x) = \exp(\beta\cos(2\pi x))
  \qquad\mbox{for all } i=0,\ldots, L-1,
\]
except that for the numerator we will replace $f_0$ by
\[
  f_0(x) = \cos(2\pi x)\, \exp(\beta\cos(2\pi x)).
\]
Note our abuse of notation here: comparing with \eqref{eq:int2} we have
the special case that $f_i(u,v) = \kappa_i(v-u) \equiv f_i(v-u)$, i.e., it
can be treated as a function of a single variable (of the difference of
the two arguments). Each function $\kappa_i \equiv f_i$ is periodic so we
know from Section~\ref{sec:first-fft} that with the rectangle rule we have
Scenario~(A7) in Table~\ref{tab1}.

Executable Julia code for this calculation is given below.

\begin{small}
\begin{verbatim}
f(beta, x)  = exp(beta * cospi(2*x))
f0(beta, x) = cospi(2*x) * f(beta, x)

function calc_U1_1d(beta::T, L::Int=10, n::Int=2^5) where {T <: AbstractFloat}
    t  = (T(0):n-1)/n                     # discretize on these points (type T)
    c  = sum( f.(beta, t)/n )             # scaling
    F  = fft( f.(beta, t)/n/c )  |> real  # Fourier coefficients of f
    F0 = fft( f0.(beta, t)/n/c ) |> real  # Fourier coefficients of f0
    Qnum = sum( F.^(L-1) .* F0 )
    Qden = sum( F.^L )
    ratio = Qnum / Qden
end
\end{verbatim}
\end{small}

Similar to Matlab, Julia allows ``broadcasting'' operations over all
elements of an array by using the ``dot syntax''. In Julia this syntax is
extended to any function by appending a dot to the function name, e.g.,
\texttt{f.(beta, t)} for a vector \texttt{t}. We will use this syntax in
Section~\ref{sec:app-QED} to do a calculation for a selection of arguments
$\beta$, $L$ and $n$. The type \texttt{T} of the parameter $\beta$ defines
the floating point type used throughout the calculation, allowing for
arbitrary precision.

\subsection{1D higher order couplings}

The action in \eqref{eq:rotor} arose from approximating a first derivative
by a forward difference formula $(x_{i+1}-x_i)/h$ of order $h$. If we use
now the central difference formula $(x_{i+1}-x_{i-1})/(2h)$ of
order~$h^2$, then we would end up with the denominator (now with $\beta =
IL^2/(4T^2)$)
\[
  \int_{D^L} \prod_{i=0}^{L-1} f\big(x_{i+1}-x_{i-1}\big)\,\rd\bsx.
\]
At first glance this appears to be a problem with order $2$ couplings, but
it can be simplified. If $L$ is even, then the variables completely
decouple into even and odd indices, and the integral can be written as the
product
\[
  \bigg(\int_{D^{L/2}} \prod_{\satop{i=0}{\rm even}}^{L-1} f(x_{i+1}-x_{i-1})\,\rd\bsx_{\rm even}\bigg)
  \bigg(\int_{D^{L/2}} \prod_{\satop{i=0}{\rm odd}}^{L-1} f(x_{i+1}-x_{i-1})\,\rd\bsx_{\rm odd}\bigg).
\]
So after reparametrization this becomes
\[
  \bigg(\int_{D^{L/2}} \prod_{j=0}^{L/2-1} f\big(y_{j+1}-y_j\big)\,\rd\bsy\bigg)^2,
\]
with $y_j \equiv y_{j\bmod (L/2)}$ for all $j\in\bbZ$, which is
essentially a first order problem. If $L$ is odd, then the variables can
be relabelled in the order of $x_0,x_2,x_4,\ldots,x_{L-1},x_1,\ldots,x_L$,
and we get
\[
  \int_{D} \prod_{j=0}^{L-1} f(y_{j+1}-y_j)\,\rd\bsy,
\]
which is exactly the first order problem. For the numerator we just need
to adjust for the different function $f_0$ but the general principle is
the same.

We can consider other higher order finite difference approximations to the
first derivative to get higher order couplings. For example, the central
difference formula $(- x_{i+2} + 8x_{i+1}-8x_{i-1}+ x_{i-2})/(12h)$ of
order~$h^4$ leads to (now with $\beta = IL^2/(144T^2)$)
\[
  \int_{D^L} \prod_{i=0}^{L-1} f\big(-x_{i+2} + 8x_{i+1}-8x_{i-1}+ x_{i-2}\big)\,\rd\bsx,
\]
which has order $r=4$ couplings and we can follow the strategy in
Section~\ref{sec:higher}. We will need to form the functions $\theta_i$ by
taking products of the functions $f_i$. We can apply Scenario (B4) of
Table~\ref{tab2} using a $4$-dimensional lattice cubature rule.

\subsection{Extension to topological susceptibility and beyond}
\label{sec:topsus}

Unfortunately, the numerical treatment of observables such as the
topological susceptibility \eqref{eq:topcharge} is a little bit more
involved than our results in Table~\ref{tab1} will indicate. The culprit
is the square in \eqref{eq:topsusc} outside the sum in
\eqref{eq:topcharge}. This formally breaks the assumed low-order coupling
structure. However, the square can be expanded into a double sum and leads
to the results of Table~\ref{tab1} needing to be applied to each of the
$L^2$ summands separately. In other words, the non-triviality of locally
defined observables capturing global properties translates into splitting
one global observable into many local observables which can then be
treated according to Table~\ref{tab1}. The topological
susceptibility~\eqref{eq:topsusc}, in particular, can be solved naively
using $L^2$ integration problems as in Scenario (A1) in Table~\ref{tab1}.
Using translational invariance of the indices, that is, the model is not
changed if each $\phi_i$ is replaced with $\phi_{i+1}$, we can furthermore
reduce the computational cost to $L$ integration problems similar to
Scenario (A4). The difference to Scenario (A4) is that now both $M_0$ and
one additional $M_i$ are different from $M$. In summary, while such global
observables may not be as efficiently solvable as purely local
observables, we only incur an overhead cost polynomial in the lattice size
$L$. The same methodology can be applied to other kinds of
susceptibilities, e.g., the magnetic susceptibility or the specific heat.

\section{Application to compact $\U(1)$ lattice gauge theory} \label{sec:app-QED}

\subsection{2D first order couplings}

The 2D compact $\U(1)$ lattice gauge theory model \eqref{eq:QED2D} takes
the generic form \eqref{eq:int2D}. We begin by separating out the
variables in the $a$-direction and the $b$-direction
\begin{align*}
  \calI
  &\,=\,
  \int_{D^{L^2}} \int_{D^{L^2}} \prod_{i=0}^{L-1} \prod_{j=0}^{L-1}
  f_{i,j}\left(x^a_{i,j} - x^a_{i,j+1} - x^b_{i,j} + x^b_{i+1,j} \right)
  \,\rd\bsx^a \,\rd\bsx^b \\
  &\,=\,
  \int_{D^{L^2}} \prod_{i=0}^{L-1}
  \bigg(
  \underbrace{
  \int_{D^L} \prod_{j=0}^{L-1}
  f_{i,j}\left(x^a_{i,j} - x^a_{i,j+1} - x^b_{i,j} + x^b_{i+1,j} \right)
  \,\rd\bsx^a_i}_{\mbox{$=:\,g_i\left(\bsx^b_{i+1}-\bsx^b_i\right)$}}
  \bigg)
  \,\rd\bsx^b,
\end{align*}
where we used the crucial fact that each factor over the index $i$ depends
only on $\bsx^a_i$, so that the integral over $\bsx^a\in D^{L^2}$ becomes
a product of the integrals over $\bsx^a_i =
(x^a_{i,0},\ldots,x^a_{i,L-1})\in D^L$. We can therefore reduce the
problem to
\begin{align} \label{eq:outer}
  \calI
   &\,=\,
  \int_{D^L} \cdots \int_{D^L}\prod_{i=0}^{L-1}
  g_i\left(\bsy_{i+1}-\bsy_i\right)
  \,\rd\bsy_0 \cdots \rd\bsy_{L-1},
\end{align}
where
\begin{align} \label{eq:inner}
  g_i(\bsy)
  &\,:=\,
  \int_{D^L} \prod_{j=0}^{L-1}
  f_{i,j}\left(x_j - x_{j+1} + y_j\right)
  \,\rd\bsx.
\end{align}

Observe that the outer integral \eqref{eq:outer} involves first order
couplings of the form \eqref{eq:intLs} with $s$ replaced by $L$, while the
inner integral \eqref{eq:inner} involves first order couplings of the form
\eqref{eq:int2} for each input $i$ and $\bsy$. Thus Table~\ref{tab1}
applies for both integrals. If all the functions $f_{i,j}$ are periodic
then so are the functions $g_i$. In this case we have Scenarios (A5)--(A7)
by using an $n$-point rectangle rule for the inner integral and an
$N$-point lattice cubature rule for the outer integral. The cost when all
functions are the same is then of the order
\[
   N\,\log(N) + N\, n\,\log(n)\,,
\]
which is independent of $L$. The error is of order $N^{-\alpha} +
n^{-\alpha}$, where $\alpha$ depends on the smoothness of the functions
and the underlying lattice rule, and the implied constant may depend
exponentially on $L$.

But more savings are possible as we explain below.

\begin{lemma} \label{lem1}
If the functions $f_{i,j}$ are periodic, then the inner integral
\eqref{eq:inner} simplifies to
\begin{align} \label{eq:inner2}
  g_i(\bsy)
  \,=\, g_i\Big(\textstyle\sum_{j=0}^{L-1}  y_j,0,\ldots,0\Big)
  \,=\, g_i\Big(\textstyle\sum_{j=0}^{L-1}  y_j,\bszero\Big),
\end{align}
that is, $g_i(\bsy)$ depends only on the sum of the components of $\bsy$.
\end{lemma}

\begin{proof}
With $\bsDelta\in \bbR^L$ to be specified later, we introduce a change of
variables $u_j = x_j + \Delta_j$ in \eqref{eq:inner} to obtain (with all
indexing to be interpreted modulo $L$)
\begin{align*}
  g_i(\bsy)
  &\,=\,
  \int_{x_{L-1}=0}^1 \cdots \int_{x_0=0}^1 \prod_{j=0}^{L-1}
  f_{i,j}\left(x_j - x_{j+1} + y_j\right)
  \,\rd x_0 \cdots \rd x_{L-1} \\
  &\,=\,
  \int_{u_{L-1}=\Delta_{L-1}}^{1+\Delta_{L-1}} \cdots \int_{u_0=\Delta_0}^{1+\Delta_0}
  \prod_{j=0}^{L-1} f_{i,j}\left(u_j - u_{j+1} + y_j - \Delta_j + \Delta_{j+1} \right)
  \,\rd u_0 \cdots \rd u_{L-1} \\
  &\,=\,
  \int_0^1 \cdots \int_0^1
  \prod_{j=0}^{L-1} f_{i,j}\left(u_j - u_{j+1} + y_j - \Delta_j + \Delta_{j+1}\right)
  \,\rd u_0 \cdots \rd u_{L-1},
\end{align*}
which follows from the periodicity of $f_{i,j}$. Now with the choice
$\Delta_0=0$, we choose $\Delta_2,\ldots,\Delta_{L-1}$ such that
\[
  y_{L-1} - \Delta_{L-1} = 0, \quad
  y_{L-2} - \Delta_{L-2} + \Delta_{L-1} = 0, \;\ldots,\quad
  y_2 - \Delta_2 + \Delta_3 = 0, \quad
  y_1 - \Delta_1 + \Delta_2 = 0.
\]
Adding these expressions together gives $\Delta_1 = \sum_{j=1}^{L-1} y_j$.
This choice of $\bsDelta$ yields
\begin{align*}
  g_i(\bsy)
  &\,=\,
  \int_0^1 \cdots \int_0^1
  f_{i,0}\Big(u_0 - u_1 + \textstyle\sum_{j=0}^{L-1} y_j\Big)
  \prod_{j=1}^{L-1} f_{i,j}\left(u_j - u_{j+1}\right)
  \,\rd u_0 \cdots \rd u_{L-1} \\
  &\,=\, g_i\Big(\textstyle\sum_{j=0}^{L-1} y_j,0,\ldots,0\Big)
  \,=\, g_i\Big(\textstyle\sum_{j=0}^{L-1} y_j,\bszero\Big).
\end{align*}
This completes the proof.
\end{proof}

\begin{lemma} \label{lem2}
If the functions $f_{i,j}$ are periodic, then the outer integral
\eqref{eq:outer} simplifies to
\begin{align} \label{eq:outer2}
   \calI \,=\,
  \int_{D^L} \prod_{i=0}^{L-1} g_i\big(y_{i+1} - y_i,\bszero\big)\,\rd\bsy.
\end{align}
\end{lemma}

\begin{proof}
Substituting \eqref{eq:inner2} into \eqref{eq:outer} we have
\begin{align*}
  \calI
   &\,=\,
  \int_{D^L} \cdots \int_{D^L}\prod_{i=0}^{L-1}
  g_i\Big(\textstyle\sum_{j=0}^{L-1} (y_{i+1,j}-y_{i,j}),\bszero\Big)
  \,\rd\bsy_0 \cdots \rd\bsy_{L-1}.
\end{align*}
We carry out a change of variables from $y_{i,0}$ to $u_i$ for each
$i=0,\ldots,L-1$ by the substitution $u_i = y_{i,0} + \sum_{j=1}^{L-1}
y_{i,j} = \sum_{j=0}^{L-1} y_{i,j}$ with Jacobian $\rd u_i = \rd y_{i,0}$,
to obtain
\begin{align*}
  \calI &\,=\,
  \int_{D^L} \cdots \int_{D^L}
  \prod_{i=0}^{L-1}
  g_i\big(u_{i+1}-u_i,\bszero\big)
  \,(\rd u_0\,\rd y_{0,1}\cdots\rd y_{0,L-1}) \cdots (\rd u_{L-1}\,\rd y_{L-1,1}\cdots\rd y_{0,L-1}) \\
  &\,=\,
  \int_D \cdots \int_D
  \prod_{i=0}^{L-1}
  g_i\big(u_{i+1}-u_i,\bszero\big)
  \,\rd u_0 \cdots \rd u_{L-1},
\end{align*}
where the remaining variables $y_{i,j}$ drop out conveniently and give us
the lower-dimensional integral in \eqref{eq:outer2}.
\end{proof}

The outer integral \eqref{eq:outer2} is now of the form \eqref{eq:int2}.
So we can apply just a rectangle rule and there is no need for a lattice
rule. The cost using the same number of points for both the inner and
outer integrals is now
\[
   n\,\log(n) + n^2\,\log(n)\,,
\]
which is again independent of $L$. The error is $\calO( n^{-\alpha})$,
where $\alpha$ depends on the smoothness of the functions.

In the context of gauge theories, Lemmas~\ref{lem1} and~\ref{lem2} can
be interpreted as a type of ``gauge fixing''. For the 2D compact $\U(1)$
lattice gauge theory model \eqref{eq:QED2D} we have $D = [0,1]$ and
\[
  f_{i,j}(x) = f(x) = \exp(\beta\cos(2\pi x))
  \qquad\mbox{for all } i,j=0,\ldots, L-1,
\]
and for the numerator we will replace $f_{0,0}$ by
\[
  f_{0,0}(x) = \cos(2\pi x)\, \exp(\beta\cos(2\pi x)).
\]
To evaluate the outer integral \eqref{eq:outer2} we need the $n\times n$
matrix $M_i$ with entries
\[
  (M_i)_{k,k'}
  \,=\, g_i\Big(\frac{(k'-k) \bmod n}{n},\bszero\Big),
  \qquad k,k'=0,\ldots,n-1.
\]
This is a circulant matrix because of the periodicity that $g_i$ inherited
from $f_{i,j}$. For the inner integrals \eqref{eq:inner2} we need to
evaluate
\begin{align*}
  g\Big(\frac{k}{n},\bszero\Big)
  \,=\, \int_{[0,1]^L}
  f\Big(x_0 - x_1 + \frac{k}{n} \Big)
  \prod_{j=1}^{L-1} f\left(x_j - x_{j+1}\right)
  \,\rd \bsx,
  \qquad k=0,\ldots,n-1.
\end{align*}
If $f_{0,0}$ is different then we also need
\begin{align*}
  g_0\Big(\frac{k}{n},\bszero\Big)
  \,=\, \int_{[0,1]^L}
  f_{0,0}\Big(x_0 - x_1 + \frac{k}{n} \Big)
  \prod_{j=1}^{L-1} f\left(x_j - x_{j+1}\right)
  \,\rd \bsx,
  \qquad k=0,\ldots,n-1.
\end{align*}
We can approximate these $2n$ values altogether using a rectangle rule
with $n$ points. All matrices will be circulant so we are in Scenario
(A7): indeed we have the situation of $B = M_0\, M^{L-1}$, where the
$n\times n$ matrix $M_0$ changes depending on the value of $k/n$, while
$M$ stays the same. This means $2n+1$ calls to FFT. The combined cost for
computing the inner integrals is then $\calO(n^2\log(n))$. These values
should be pre-computed and stored. For the outer integral we are again in
Scenario~(A7) so this can be computed with cost $\calO(n\log(n))$. The
overall cost is then of the order $n\log(n) + n^2\log(n)$ as claimed.

Executable Julia code for the 2D compact $\U(1)$ lattice gauge theory model
is given below.

\begin{small}
\ifdefined\journalstyle 
\begin{verbatim}
f(beta, x)  = exp(beta * cospi(2*x))
f0(beta, x) = cospi(2*x) * f(beta, x)

function calc_U1_2d_a(beta::T, L::Int=10, n::Int=2^5, N::Int=n) where {T <: AbstractFloat}
    ## inner integral
    t = reshape((T(0):n-1)/n, n, 1)                  # n-by-1: column vector of type T
    k = reshape(T(0):N-1,     1, N)                  # 1-by-N: row vector of type T
    c = sum( f.(beta, t)/n )                         # scaling
    # g
    F = fft( f.(beta, t)/n/c        , 1 )   |> real  # n-by-1 eigenvalues of circulant
    K = fft( f.(beta, t .+ k/N)/n/c , 1 )   |> real  # n-by-N eigenvalues of circulant
    g = sum( K .* F.^(L-1) , dims=1 )                # 1-by-N (value for every k-value)
    # g0
    K0 = fft( f0.(beta, t .+ k/N)/n/c , 1 ) |> real  # n-by-N eigenvalues of circulant
    g0 = sum( K0 .* F.^(L-1) , dims=1 )              # 1-by-N (value for every k-value)
    ## outer integral
    G  = fft( g/N  , 2 ) |> real
    G0 = fft( g0/N , 2 ) |> real
    Qnum = sum( G0 .* G.^(L-1) )
    Qden = sum( G.^L )
    ## result
    ratio = Qnum / Qden # the scaling in both numerator and denominator cancel
end
\end{verbatim}
\else
\begin{verbatim}
f(beta, x)  = exp(beta * cospi(2*x))
f0(beta, x) = cospi(2*x) * f(beta, x)

function calc_U1_2d_a(beta::T, L::Int=10, n::Int=2^5, N::Int=n)
                                                            where {T <: AbstractFloat}
    ## inner integral
    t = reshape((T(0):n-1)/n, n, 1)                  # n-by-1: column vector of type T
    k = reshape(T(0):N-1,     1, N)                  # 1-by-N: row vector of type T
    c = sum( f.(beta, t)/n )                         # scaling
    # g
    F = fft( f.(beta, t)/n/c        , 1 )   |> real  # n-by-1 eigenvalues of circulant
    K = fft( f.(beta, t .+ k/N)/n/c , 1 )   |> real  # n-by-N eigenvalues of circulant
    g = sum( K .* F.^(L-1) , dims=1 )                # 1-by-N (value for every k-value)
    # g0
    K0 = fft( f0.(beta, t .+ k/N)/n/c , 1 ) |> real  # n-by-N eigenvalues of circulant
    g0 = sum( K0 .* F.^(L-1) , dims=1 )              # 1-by-N (value for every k-value)
    ## outer integral
    G  = fft( g/N  , 2 ) |> real
    G0 = fft( g0/N , 2 ) |> real
    Qnum = sum( G0 .* G.^(L-1) )
    Qden = sum( G.^L )
    ## result
    ratio = Qnum / Qden # the scaling in both numerator and denominator cancel
end
\end{verbatim}
\fi
\end{small}

It turns out that we can use an alternative approach based on Fourier
series to simplify the expression so that there is actually no need for
nested integral calculations.

\begin{theorem} \label{thm}
Suppose that the functions $f_{i,j}$ are periodic and have absolutely
convergent Fourier series. Define $\mu_{i+jL} := f_{i,j}$ for
$i,j=0,\ldots,L-1$. Then the integral \eqref{eq:outer}, with inner integral \eqref{eq:inner},
simplifies to
\begin{align} \label{eq:outer3}
   \calI \,=\,
  \int_{D^{L^2}} \prod_{k=0}^{L^2-1} \mu_\ell\big(x_{k+1}-x_\ell\big)\,\rd\bsx,
\end{align}
where now the parametric periodicity is to be taken modulo $L^2$, i.e.,
$x_k \equiv x_{k\bmod L^2}$.
\end{theorem}

\begin{proof}
From Lemma~\ref{lem:conv2D} in \RefApp{app:fourier} we know that the
integral for the 2D problem can be written in terms of the Fourier
coefficients of $f_{i,j}$ as
\begin{align*}
  \calI \,=\, \sum_{\ell\in\bbZ} \prod_{i=0}^{L-1} \prod_{j=0}^{L-1} \widehat{f_{i,j}}(\ell).
\end{align*}
With the relabeling of the functions $\mu_{i+jL} := f_{i,j}$, we can
rewrite the above sum as
\begin{align*}
  \calI \,=\, \sum_{\ell\in\bbZ} \prod_{k=0}^{L^2-1} \widehat{\mu_k}(\ell).
\end{align*}
Comparing with the 1D problem in Lemma~\ref{lem:conv1D} in
\RefApp{app:fourier}, we conclude that this sum can be rewritten as
an integral over $D^{L^2}$ as shown in \eqref{eq:outer3}, with parametric
periodicity modulo $L^2$.
\end{proof}

Theorem~\ref{thm} means that we do not have nested integrals any more, but
instead we have a new integral with dimensionality $L^2$. We have just one
integral of the form \eqref{eq:int2}, with $L$ replaced by $L^2$, so we
are again in Scenario~(A7). The cost is only of order
\[
  n\,\log(n),
\]
and the error is $\calO(n^{-\alpha})$.

Hence we can use the 1D rotor code to calculate the 2D compact $\U(1)$
lattice gauge theory simply by
\begin{small}
\begin{verbatim}
calc_U1_2d_b(beta, L=10, n=2^5) = calc_U1_1d(beta, L^2, n)
\end{verbatim}
\end{small}

We illustrate the code by a small numerical experiment which calculates
some values for the 2D compact $\U(1)$ lattice gauge theory:
\begin{small}
\begin{verbatim}
# calculate for each beta, L and n in the following three lists by using broadcasting
step = BigFloat("0.1")  # use of arbitrary precision type BigFloat (optional)
# step = 0.1            # alternative: if wanting IEEE double just uncomment this line
betas = 0:step:10                            # size 101 (for the given step)
Ls = [2, 20, 200]'                           # size 1-by-3
ns = reshape([2^4, 2^6, 2^8, 2^10], 1, 1, 4) # size 1-by-1-by-4

# do all calculations:
X2 = calc_U1_2d_b.(betas, Ls, ns)          # 101-by-3-by-4 result array by broadcasting

# print the values for L = 200 and n = 2^10 for some values of beta:
for i=2:10:length(betas)
    println(Float64(betas[i]), " ", X2[i,end,end])
end
\end{verbatim}
\end{small}
which prints to approximately 79~decimal digits the following values
\begin{small}
\begin{verbatim}
0.1 0.04993760398793891942505492702790735280024819495932643969025083229259197970124841
1.1 0.4807027720204957075397353534961410739293237985698753220914923708183899597383392
2.1 0.7135313929252366606474906234333206952579818112136755308698717366991034508513433
3.1 0.8171145492914306407729604696551455026259470380147213328440139033655041292524231
4.1 0.8672601961768063107300630399509515633441383106454204305046785090286863340013323
5.1 0.895651587990760146226062237096125294781978743841561112098097135505392751620628
6.1 0.9138858516725660997721369731593268002734144795505884713207423774281830098389897
7.1 0.9266326601661551618966214804622090172117923146860230737466143933956830061458322
8.1 0.9360676059396539968069515062581218179367309351114611124563259272640695892107535
9.1 0.9433416321068225957542493497236464198235930555179598879914504004185617655236025
\end{verbatim}
\end{small}

\begin{figure}
  \includegraphics[width=0.5\textwidth]{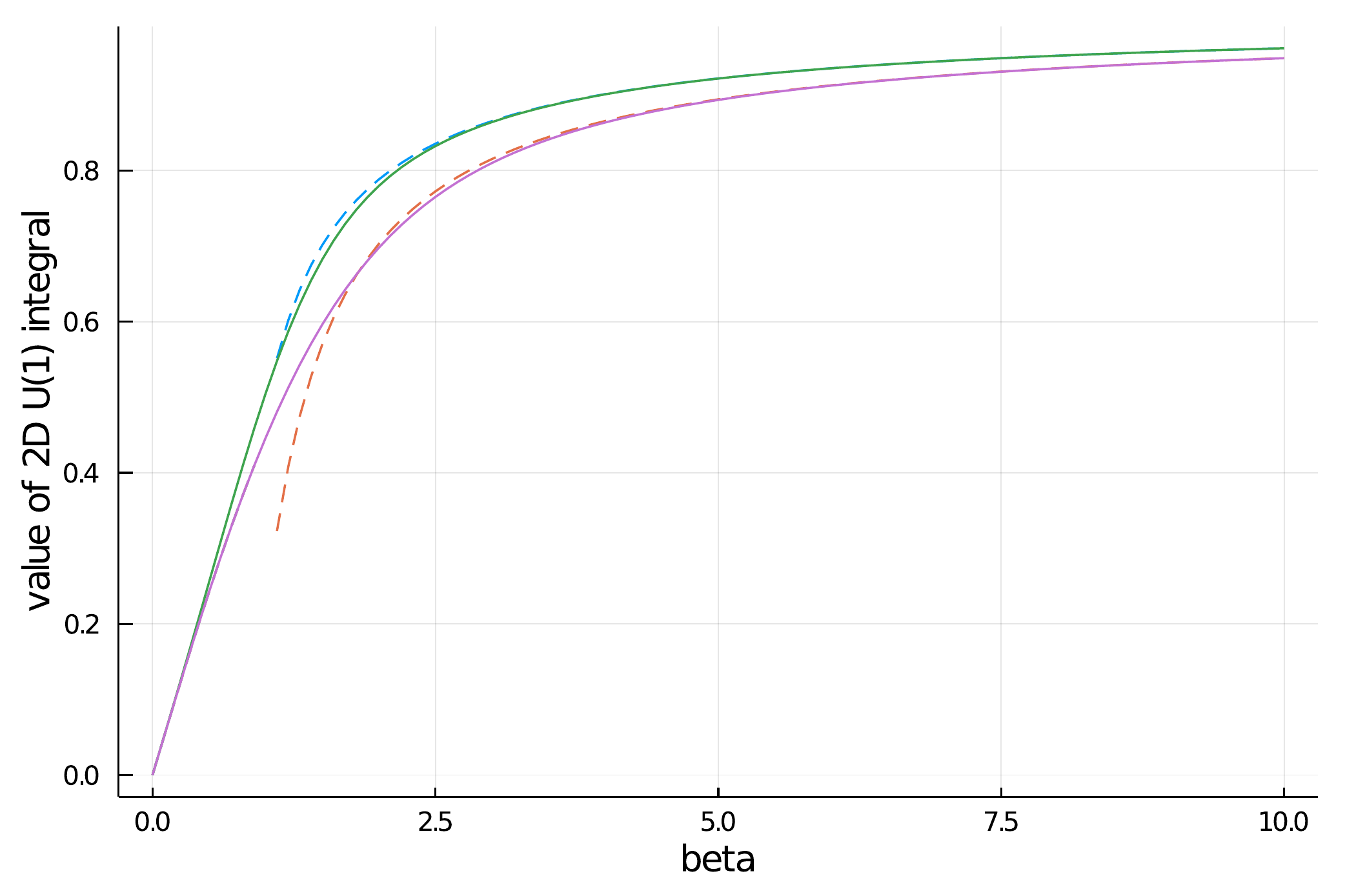}
  \includegraphics[width=0.5\textwidth]{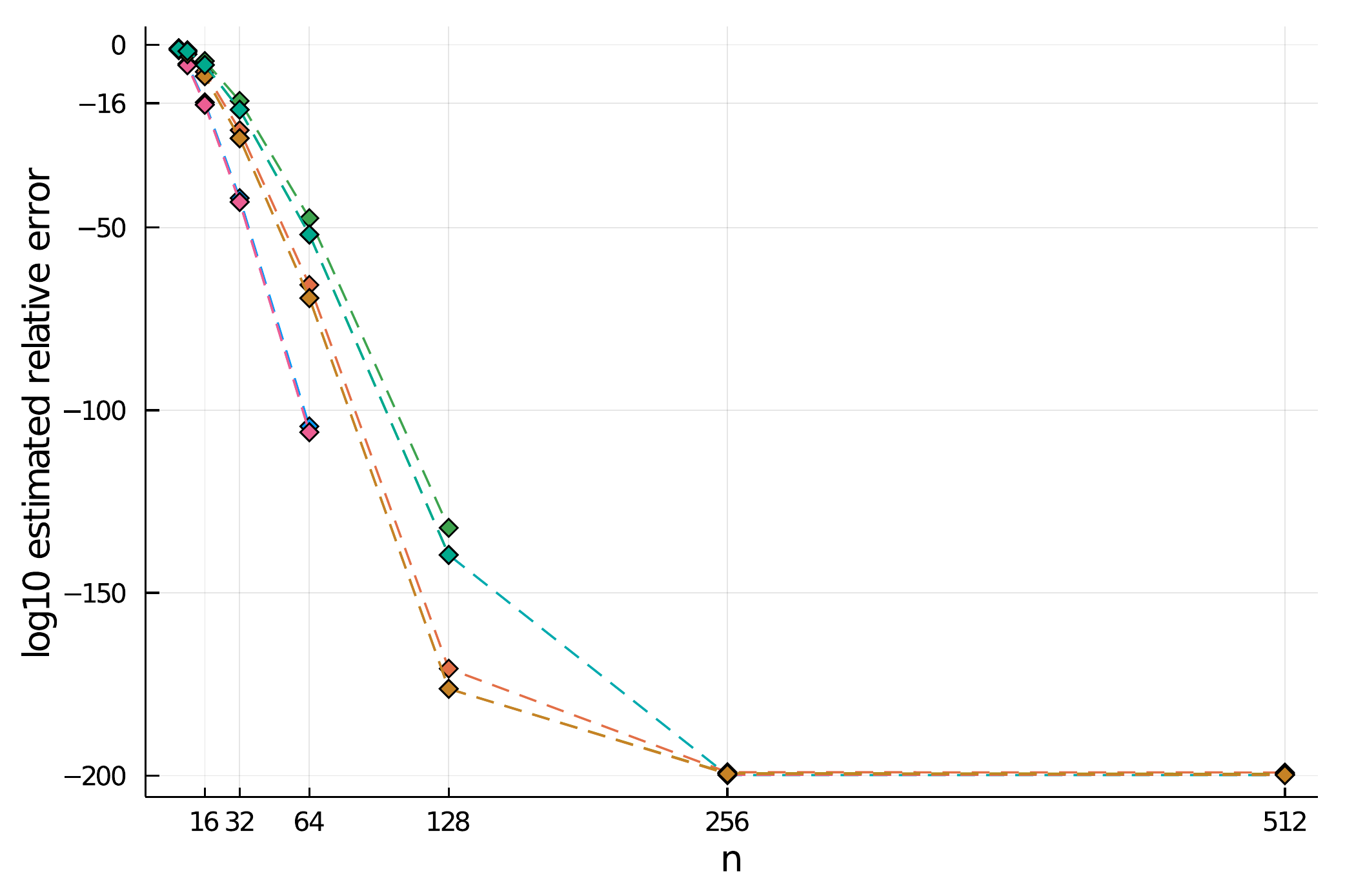}
  \caption{
  Illustration of numerical results for the 2D compact $U(1)$ lattice gauge theory.
  Left: Comparison of our FFT based algorithm (solid lines) with two known asymptotical formulae for $\beta \ll 1$ and $\beta \gg 1$ (dashed lines); the top curve is for $L=2$ and the bottom curve is for $L=200$.
  Right: Estimated accuracy showing exponential convergence for increasing $n$ compared to our final calculated value with $n=1024$ for all combinations of $L \in \{2,20,200\}$ and $\beta \in \{1, 4, 8\}$ (see text).}\label{fig:U1:2d}
\end{figure}

In the left hand side of Figure~\ref{fig:U1:2d} we compare our calculated values to a known asymptotic formula for $\beta \ll 1$ from \cite{Balian:1974xw} and $\beta \gg 1$ from~\cite{Horsley:1981gj} for the 2D compact $U(1)$ lattice gauge integral for $L=2$ and $L=200$.
We can see that for small values of $\beta$ and for large values of $\beta$ the asymptotic formulae are close to our calculated values, while in the neighbourhood of $\beta = 1$ the asymptotic formulae deviate more, as expected.

In the right hand side of Figure~\ref{fig:U1:2d} we illustrate the accuracy of our method in terms of $n$.
To calculate a reference value for different values of $L$ and $\beta$ we have run our code with $n = 1024$ and increased the precision of Julia's \texttt{BigFloat} type to approximately $200$ decimal digits.
We plot the base~$10$ logarithm of the estimated relative error (using the reference value for $n=1024$) in terms of $n$ on the horizontal axis in linear scale.
We plot the error for all combinations of $L \in \{2,20,200\}$ and $\beta \in \{1, 4, 8\}$.
The missing data points are when the error was zero.
Since $200$ decimal digits is our maximum precision the lines will flat line at $-200$.
We observed that the lines for the different values of $L$ are always close together and hence we conclude that the value of $L$ does not really matter for the performance of the algorithm.
For the different choices of $\beta$ we observed that larger values of $\beta$ make the problem slightly harder. E.g., the lines for $\beta = 8$ are the ones on top with the slowest decay, while the lines for $\beta = 1$ are the ones at the bottom with the fastest decay.
This is not a surprise, and obviously, the limiting case of $\beta = 0$ is the trivial problem.
All cases show exponential convergence, this is why we resorted to arbitrary precision to make the graph.
If one is only interested in double precision results, i.e., a relative error of about $10^{-16}$, then we see from the graph that $n = 32$ is sufficient to give $16$ decimal digits for the large $\beta = 8$, while $n = 16$ would suffice for $\beta = 1$.
From the graph we also see that with $n$ somewhere between $128$ and $256$ we can expect to have more than $200$ digits decimal precision for any of the $\beta$'s and $L$'s that we plotted.

Calculating all the results for the $101$ $\beta$ values and $3$ $L$
values from the Julia code snippet above in double precision with $n = 32$
takes less than $10$~milliseconds on a MacBook Pro from 2013 (2.6~Ghz
Dual-Core Intel Core i5).

\subsection{2D higher order couplings}

The denominator in the Wilson loop \eqref{eq:Wilson} can be expressed as
\begin{align*}
  &
  \int_{D^{L^2}} \int_{D^{L^2}} \prod_{i=0}^{L-1}
  \underbrace{\prod_{j=0}^{L-1}
  f_{i,j}\Big( \textstyle\sum_{k=0}^{r_a} \big(x^a_{i+k,j}- x^a_{i+k,j+r_b}\big)
  + \textstyle\sum_{k=0}^{r_b} \big(x^b_{i+r_a,j+k}- x^b_{i,j+k}\big) \Big)
  }_{=:\, g_i\big( (\bsx_i^a, \bsx_{i+1}^a, \ldots,\bsx_{i+r_a}^a);\,\bsx^b \big)}
   \,\rd\bsx^a \,\rd\bsx^b \\
  &\,=\,
  \int_{D^{L^2}} \bigg( \int_{D^{L^2}} \prod_{i=0}^{L-1}
  g_i\Big( (\bsx_i^a, \bsx_{i+1}^a, \ldots,\bsx_{i+r_a}^a); \bsx^b \Big)
   \,\rd\bsx^a \bigg) \rd\bsx^b
   .
\end{align*}
Assuming that $L$ is a multiple of $r_a$, the inner integral has order
$r_a$ coupling and can be turned into an $L/r_a$-fold product of
$r_aL$-dimensional integrals following Section~\ref{sec:higher}. We can
use an $N$-point lattice rule in $r_aL$ dimensions as in Scenario~(B4) at
the cost of order $(L/r_a)\,N^3$ times the number of different samples of
$\bsx^b$. The outer integral over $\bsx^b$ cannot be simplified so it is
of dimensionality $L^2$. This problem is truly high dimensional, except
for the special case $r_a = r_b = 1$ which can be simplified as we show in
Corollary~\ref{cor:Wilson} in \RefApp{app:fourier} using Fourier
series.

\subsection{3D first order couplings}

The 3D compact $\U(1)$ lattice gauge theory problem \eqref{eq:QED3D} is
very tough. We show in Lemma~\ref{lem:conv3D} in
\RefApp{app:fourier} that there is an explicit expression in terms
of Fourier coefficients. Further work is needed to see how this expression
can be used to simplify the integral calculations.

\section{Summary} \label{sec:summary}

In this paper we developed efficient recursive strategies to tackle a
class of high dimensional integrals having a special product structure
with low order couplings, motivated by physics models such as the quantum
rotor and the 2D compact $U(1)$ lattice gauge theory
(Section~\ref{sec:physics}). We reviewed and extended the recursive
strategy from \cite{AGHJLV16,Hartung:2020uuj} for generic integrals with
first order couplings (not necessarily from physics) to identify scenarios
that enable the use of FFT for efficient computation as well as the use of
lattice cubature rules when we have an $L$-fold product of $s$-dimensional
integrals (Section~\ref{sec:first}). Furthermore, we extended the
recursive strategy to higher order couplings, noting that the problems can
become truly high dimensional (Section~\ref{sec:higher}). Then we
considered particular physics applications (Sections~\ref{sec:app-rotor}
and~\ref{sec:app-QED}) and provided Julia codes for the special cases of
quantum rotor and 2D compact $U(1)$ lattice gauge theory. Finally we
provided an alternative formulation of the integrals in terms of Fourier
series to pave the way for future work for tough 2D and 3D physics problems
(Appendix).

\paragraph{Acknowledgements}
We gratefully acknowledge the financial support from
the Australian Research Council under grant DP180101356 and
the Research Foundation Flanders (FWO) under grant G091920N.

\appendix

\section{Alternative approach via Fourier series} \label{app:fourier}

\subsection{1D problems}

\begin{lemma} \label{lem:conv1D}
Let $f_i:D\to\bbR$ be periodic and have an absolutely convergent Fourier
series, and assume parametric periodicity $x_i = x_{i\bmod L}$. Then
\begin{align*}
  \int_{D^L}\prod_{i=0}^{L-1} f_i\big(x_{i+1}-x_i\big) \,\rd\bsx
  \,=\, \sum_{\ell\in\bbZ} \prod_{i=0}^{L-1} \widehat{f_i}(\ell).
\end{align*}
\end{lemma}

\begin{proof}
Define a periodic function
\begin{align*}
  \calI(\bsy)
  &:=\,
  \int_{D^L}\prod_{i=0}^{L-1} f_i\big(\xi_i(\bsx) + y_i \big)\,\rd\bsx,
  \qquad \bsy\in D^L,
\end{align*}
for generic functions $\xi_i$, and consider its Fourier series
\[
  \calI(\bsy) \,=\, \sum_{\bsh\in\bbZ^L} \widehat{\calI}(\bsh)\,e^{2\pi\ri\,\bsh\cdot\bsy},
  \qquad
  \widehat{\calI}(\bsh) \,:=\,
  \int_{D^L} \calI(\bsy)\, e^{-2\pi\ri\,\bsh\cdot\bsy}
  \,\rd\bsy.
\]
The desired integral is recovered by evaluating the Fourier series at
$\bsy = \bszero$: $\calI(\bszero) \,=\, \sum_{\bsh\in\bbZ^L}
\widehat{\calI}(\bsh)$.

We proceed to compute the Fourier coefficients $\widehat{\calI}(\bsh)$.
Due to the product structure, all integrals in $\widehat{\calI}(\bsh)$ are
one-dimensional
\begin{align} \label{eq:fourier}
  &\int_D f_i\big(\xi_i(\bsx) + y_i \big) \, e^{-2\pi\ri\,h_i\,y_i}\,\rd y_i
  \,=\, \int_D \sum_{\ell\in\bbZ} \widehat{f_i}(\ell)\,e^{2\pi\ri\,\ell\,(\xi_i(\bsx) + y_i)}
  \,e^{-2\pi\ri\,h_i\,y_i}\,\rd y_i \nonumber \\
  &\,=\, \sum_{\ell\in\bbZ} \widehat{f_i}(\ell)\,e^{2\pi\ri\,\ell\,\xi_i(\bsx)}
  \int_D e^{2\pi\ri\,(\ell - h_i)\,y_i}\,\rd y_i
  \,=\, \widehat{f_i}(h_i)\,e^{2\pi\ri\,h_i\,\xi_i(\bsx)},
\end{align}
where we used $\int_D e^{2\pi\ri (\ell-h)y}\,\rd y = 1$ if $\ell=h$ and is
$0$ otherwise. Thus
\begin{align*}
  \widehat{\calI}(\bsh)
  &\,=\,
  \int_{D^L} \prod_{i=0}^{L-1} \widehat{f_i}(h_i)\,e^{2\pi\ri\,h_i\,\xi_i(\bsx)}\,\rd\bsx
  \,=\,
  \bigg(\prod_{i=0}^{L-1} \widehat{f_i}(h_i) \bigg)
  \int_{D^L} e^{2\pi\ri\, \sum_{i=0}^{L-1} h_i\,\xi_i(\bsx)}\,\rd\bsx.
\end{align*}

Specializing now to $\xi_i(\bsx) := x_{i+1} - x_i$, we have for the
exponent
\[
 \sum_{i=0}^{L-1} h_i\,\xi_i(\bsx)
 \,=\, \sum_{i=0}^{L-1} h_i\,x_{i+1} - \sum_{i=0}^{L-1} h_i\,x_i
 \,=\, \sum_{i=0}^{L-1} h_{i-1}\,x_i - \sum_{i=0}^{L-1} h_i\,x_i
 \,=\, \sum_{i=0}^{L-1} (h_{i-1} - h_i)\,x_i,
\]
where in the second equality we re-indexed the first sum using the
property that all indices are taken modulo $L$. Thus $\int_{D^L}
e^{2\pi\ri\, \sum_{i=0}^{L-1} h_i\,\xi_i(\bsx)}\,\rd\bsx = 1$ if and only
if
\begin{align} \label{eq:cond1D}
  h_{i-1} = h_i \qquad\mbox{for all $i$ (taken modulo $L$)},
\end{align}
and the integral is zero otherwise. We conclude that all components of
$\bsh$ must be the same for the corresponding Fourier coefficient to be
nonzero. Hence
\[
  \calI(\bsy) \,=\, \sum_{\ell\in\bbZ} \bigg(\prod_{i=0}^{L-1} \widehat{f_i}(\ell)\bigg)
  \,e^{2\pi\ri\,(\ell,\ldots,\ell)\cdot\bsy}.
\]
Our desired integral is recovered by evaluating the Fourier series at
$\bsy = \bszero$.
\end{proof}

The result can be extended to higher order couplings by changing the
definition of the functions $\xi_i$ in the proof.

\begin{corollary}
Let $f_i:D\to\bbR$ be periodic and have an absolutely convergent Fourier
series, and assume parametric periodicity $x_i = x_{i\bmod L}$. Then
\begin{align*}
  \int_{D^L}\prod_{i=0}^{L-1} f_i\big(x_{i+1}-x_{i-1}\big) \,\rd\bsx
  \,=\,
  \begin{cases}
  \displaystyle\sum_{\ell\in\bbZ} \prod_{i=0}^{L-1} \widehat{f_i}(\ell) & \mbox{if $L$ is odd}, \\
  \displaystyle
  \bigg(\sum_{\ell\in\bbZ} \prod_{\satop{i=0}{\rm even}}^{L-1} \widehat{f_i}(\ell)\bigg)
  \bigg(\sum_{\ell\in\bbZ} \prod_{\satop{i=0}{\rm odd}}^{L-1} \widehat{f_i}(\ell)\bigg)
  & \mbox{if $L$ is even}.
  \end{cases}
\end{align*}
\end{corollary}

\begin{proof}
We replace the functions $\xi_i$ in the proof of Lemma~\ref{lem:conv1D} by
$\xi_i(\bsx) := x_{i+1} - x_{i-1}$. Then the condition \eqref{eq:cond1D}
becomes
\[
  h_{i-1} = h_{i+1} \qquad\mbox{for all $i$ (taken modulo $L$)}.
\]
We conclude that if $L$ is odd then all components of $\bsh$ must be
equal, and if $L$ is even then there are two possible values for the
components of $\bsh$ depending on whether the index is even or odd. This
leads to the corollary.
\end{proof}

\subsection{2D problems}

The same strategy can be used to tackle 2D problems.

\begin{lemma} \label{lem:conv2D}
Let $f_{i,j}:D\to\bbR$ be periodic and have an absolutely convergent
Fourier series, and assume parametric periodicity modulo $L$. Then
\begin{align*}
  \int_{D^{L^2}} \int_{D^{L^2}} \prod_{i=0}^{L-1} \prod_{j=0}^{L-1}
  f_{i,j}\big(x^a_{i,j} - x^a_{i,j+1} - x^b_{i,j} + x^b_{i+1,j}\big) \,\rd\bsx^a\,\rd\bsx^b
  \,=\, \sum_{\ell\in\bbZ} \prod_{i=0}^{L-1} \prod_{j=0}^{L-1} \widehat{f_{i,j}}(\ell).
\end{align*}
\end{lemma}

\begin{proof}
As in the 1D problem we define
\begin{align*}
  \calI(\bsy)
  &\,:=\,
  \int_{D^{L^2}} \int_{D^{L^2}} \prod_{i=0}^{L-1} \prod_{j=0}^{L-1}
  f_{i,j}\big(\xi_{i,j}(\bsx^a,\bsx^b) + y_{i,j} \big)
  \,\rd\bsx^a\,\rd\bsx^b,
  \qquad \bsy\in D^{L^2},
\end{align*}
for generic functions $\xi_{i,j}$, and consider the Fourier series
\[
  \calI(\bsy) \,=\, \sum_{\bsh\in\bbZ^{L^2}} \widehat{\calI}(\bsh)\,e^{2\pi\ri\,\bsh\cdot\bsy},
  \qquad
  \widehat{\calI}(\bsh) \,:=\,
  \int_{D^{L^2}} \calI(\bsy)\, e^{-2\pi\ri\,\bsh\cdot\bsy}
  \,\rd\bsy,
\]
where for the dot product we interpret an element in $D^{L^2}$ as a vector
of length $L^2$ rather than as a matrix of $L\times L$. Our desired
integral is recovered by evaluating the Fourier series at $\bsy =
\bszero$.

Analogously to \eqref{eq:fourier},
all integrals in $\widehat{\calI}(\bsh)$ are one-dimensional
\begin{align*}
  &\int_D f_{i,j}\big(\xi_{i,j}(\bsx^a,\bsx^b) + y_{i,j} \big)
  \, e^{-2\pi\ri\,h_{i,j}\,y_{i,j}}\,\rd y_{i,j}
  \,=\, \widehat{f_{i,j}}(h_{i,j})\,
  e^{2\pi\ri\,h_{i,j}\,\xi_{i,j}(\bsx^a,\bsx^b)},
\end{align*}
and thus
\begin{align} \label{eq:hat2D}
  \widehat{\calI}(\bsh)
  &\,=\,
  \bigg(\prod_{i=0}^{L-1} \prod_{j=0}^{L-1} \widehat{f_{i,j}}(h_{i,j}) \bigg)
  \int_{D^{L^2}} \int_{D^{L^2}} e^{2\pi\ri\, \sum_{i=0}^{L-1} \sum_{j=0}^{L-1} h_{i,j}\,\xi_{i,j}(\bsx^a,\bsx^b)}
  \,\rd\bsx^a\,\rd\bsx^b.
\end{align}
Specializing now to $\xi_{i,j}(\bsx^a,\bsx^b) \,:=\, x^a_{i,j} -
x^a_{i,j+1} - x^b_{i,j} + x^b_{i+1,j}$, we have the exponent
\begin{align*}
  \sum_{i=0}^{L-1} \sum_{j=0}^{L-1} h_{i,j}\,\xi_{i,j}(\bsx^a,\bsx^b)
  &\,=\, \sum_{i=0}^{L-1} \sum_{j=0}^{L-1} \Big[ \big(h_{i,j} - h_{i,j-1}\big) \,x^a_{i,j} -
  \big(h_{i,j} - h_{i-1,j}\big)\,x^b_{i,j} \Big],
\end{align*}
where
we re-indexed some terms since all indices should be taken modulo $L$. We
conclude that
\begin{align*}
  &\int_{D^{L^2}} \int_{D^{L^2}} e^{2\pi\ri\, \sum_{i=0}^{L-1} \sum_{j=0}^{L-1} h_{i,j}\,\xi_{i,j}(\bsx^a,\bsx^b)}
  \,\rd\bsx^a\,\rd\bsx^b \\
  &\,=\,
  \prod_{i=0}^{L-1} \prod_{j=0}^{L-1} \bigg[
  \bigg(\int_D e^{2\pi\ri (h_{i,j} - h_{i,j-1})\,x^a_{i,j}}\,\rd x^a_{i,j}\bigg)
  \bigg(\int_D e^{-2\pi\ri (h_{i,j} - h_{i-1,j}) \,x^b_{i,j}}\,\rd x^b_{i,j} \bigg)
  \bigg],
\end{align*}
which is equal to $1$ if and only if
\begin{align*}
  h_{i,j} = h_{i,j-1} = h_{i-1,j}
  \qquad\mbox{for all $i,j$ (taken modulo $L$)},
\end{align*}
and the integral is equal to $0$ otherwise. This means that all components
of $\bsh$ are equal, and we have reduced $L^2$ parameters down to $1$.
This yields the desired formula.
\end{proof}

The result extends trivially to the Wilson loop with $r_a = r_b = 1$.

\begin{corollary} \label{cor:Wilson}
Let $f_{i,j}:D\to\bbR$ be periodic and have an absolutely convergent
Fourier series, and assume parametric periodicity modulo $L$. The Wilson
loop with parameters $r_a$ and $r_b$ requires
\begin{align*}
  \calI^{r_a,r_b}
  &\,:=\, \int_{D^{L^2}} \int_{D^{L^2}} \prod_{i=0}^{L-1} \prod_{j=0}^{L-1}
  f_{i,j}\big(\xi_{i,j}(\bsx^a,\bsx^b)\big)
  \,\rd\bsx^a\,\rd\bsx^b, \\
  \xi_{i,j}(\bsx^a,\bsx^b) &\,:=\,
  x^a_{i,j} + x^a_{i+1,j} + \cdots + x^a_{i+r_a,j}
  + x^b_{i+r_a,j} + x^b_{i+r_a,j+1} + \cdots + x^b_{i+r_a,j+r_b} \\
  &\qquad - x^a_{i+r_a,j+r_b} - x^a_{i+r_a-1,j+r_b} - \cdots - x^a_{i,j+r_b}
  - x^b_{i,j+r_b} - x^b_{i,j+r_b-1} - \cdots - x^b_{i,j}.
\end{align*}
We have
\[
 \calI^{1,1}
 \,=\, \begin{cases}
 \displaystyle
 \sum_{\ell\in\bbZ}
  \prod_{i=0}^{L-1} \prod_{j=0}^{L-1} \widehat{f_{i,j}}(\ell)
  & \mbox{if $L$ is odd}, \\
  \displaystyle
  \bigg(\sum_{\ell\in\bbZ}
  \underbrace{\textstyle\prod_{i=0}^{L-1} \prod_{j=0}^{L-1}}_{i+j~\mathrm{odd}} \widehat{f_{i,j}}(\ell) \bigg)
  \bigg(\sum_{\ell\in\bbZ}
  \underbrace{\textstyle\prod_{i=0}^{L-1} \prod_{j=0}^{L-1}}_{i+j~\mathrm{even}} \widehat{f_{i,j}}(\ell) \bigg)
  & \mbox{if $L$ is even}.
  \end{cases}
\]
\end{corollary}

\begin{proof}
Following the proof of Lemma~\ref{lem:conv2D}, the exponent
$\sum_{i=0}^{L-1} \sum_{j=0}^{L-1} h_{i,j}\,\xi_{i,j}(\bsx^a,\bsx^b)$ is
now
\begin{align*}
  &
  \sum_{i=0}^{L-1} \sum_{j=0}^{L-1} \bigg[
  \sum_{k=0}^{r_a} \big( h_{i,j}\,x_{i+k,j}^a - h_{i,j}\,x_{i+k,j+r_b}^a \big)
  + \sum_{k=0}^{r_b} \big( h_{i,j}\,x_{i+r_a,j+k}^b - h_{i,j}\,x_{i,j+k}^b \big) \bigg] \\
  &\,=\, \sum_{i=0}^{L-1} \sum_{j=0}^{L-1} \bigg[
  \sum_{k=0}^{r_a} \big( h_{i-k,j} - h_{i-k,j-r_b} \big)\,x_{i,j}^a
  + \sum_{k=0}^{r_b} \big( h_{i-r_a,j-k} - h_{i,j-k}\big)\,x_{i,j+k}^b  \bigg],
\end{align*}
where we again re-indexed some terms. We conclude that the integral in
\eqref{eq:hat2D}
is equal to $1$ if and only if for all $i,j$ (taken modulo $L$)
\begin{align} \label{eq:wilson}
  \sum_{k=0}^{r_a} h_{i-k,j} \,=\, \sum_{k=0}^{r_a}  h_{i-k,j-r_b}
  \qquad\mbox{and}\qquad
  \sum_{k=0}^{r_b} h_{i,j-k} \,=\, \sum_{k=0}^{r_b} h_{i-r_a,j-k},
\end{align}
and is $0$ otherwise.
For the special case $r_a = r_b = 1$, the conditions in \eqref{eq:wilson}
are
\[
 \begin{cases}
  h_{i,j} + h_{i-1,j} \,=\, h_{i,j-1} + h_{i-1,j-1}, \\
  h_{i,j} + h_{i,j-1} \,=\, h_{i-1,j} + h_{i-1,j-1} .
  \end{cases}
\]
Adding and subtracting these two expressions lead to, respectively,
\[
  h_{i,j} = h_{i-1,j-1}
  \qquad\mbox{and}\qquad
  h_{i-1,j} = h_{i,j-1}.
\]
If $L$ is a multiple of $2$, then we conclude that the components of
$\bsh$ can take only two possible values depending on the value of
$(i+j)\bmod 2$, following a chessboard pattern. On the other hand, if $L$
is not a multiple of $2$ then all components of $\bsh$ have the same
value. These lead to the formulas in the corollary.
\end{proof}

\subsection{3D problems}

\begin{lemma} \label{lem:conv3D}
Let $f_{i,j,k}:D\to\bbR$ be periodic and have an absolutely convergent
Fourier series, and assume parametric periodicity modulo $L$. Then
\begin{align*}
  &\int_{D^{3L^3}} \prod_{i=0}^{L-1} \prod_{j=0}^{L-1} \prod_{k=0}^{L-1} \Big[
  f_{i,j,k}\big(x^a_{i,j,k} - x^a_{i,j+1,k} - x^b_{i,j,k} + x^b_{i+1,j,k} \big) \\
  &\qquad\qquad\qquad\qquad
  \cdot
  f_{i,j,k}\big(x^c_{i,j,k} - x^c_{i+1,j,k} - x^a_{i,j,k} + x^a_{i,j,k+1} \big) \\
  &\qquad\qquad\qquad\qquad
  \cdot
  f_{i,j,k}\big(x^b_{i,j,k} - x^b_{i,j,k+1} - x^c_{i,j,k} + x^c_{i,j+1,k} \big)
  \Big]
  \,\rd\bsx \\
  &\,=\,
  \sum_{\bsh\in\calH}
  \prod_{i=0}^{L-1} \prod_{j=0}^{L-1} \prod_{k=0}^{L-1}
  \Big(\widehat{f_{i,j,k}}(h^a_{i,j,k})\,\widehat{f_{i,j,k}}(h^b_{i,j,k})\,\widehat{f_{i,j,k}}(h^c_{i,j,k})\Big),
\end{align*}
where $\bsh = (\bsh^a,\bsh^b,\bsh^c)\in\calH \subset \bbZ^{3L^3}$
satisfies for all $i,j,k$ modulo $L$,
\begin{align} \label{eq:system}
  \begin{cases}
  h^c_{i,j,k} - h^c_{i,j-1,k} - h^b_{i,j,k} + h^b_{i,j,k-1} \,=\, 0, \\
  h^a_{i,j,k} - h^a_{i,j,k-1} - h^c_{i,j,k} + h^c_{i-1,j,k} \,=\, 0, \\
  h^b_{i,j,k} - h^b_{i-1,j,k} -  h^a_{i,j,k} + h^a_{i,j-1,k} \,=\, 0.
  \end{cases}
\end{align}
\end{lemma}

\begin{proof}
Generalizing the 1D and 2D arguments, we define for $\bsy =
(\bsy^a,\bsy^b,\bsy^c) \in D^{3L^3}$,
\begin{align*}
  \calI(\bsy)
  &\,:=\,
  \int_{D^{3L^3}} \prod_{i=0}^{L-1} \prod_{j=0}^{L-1} \prod_{k=0}^{L-1} \Big[
  f_{i,j,k}\big(x^a_{i,j,k} - x^a_{i,j+1,k} - x^b_{i,j,k} + x^b_{i+1,j,k} + y^c_{i,j,k} \big) \\
  &\qquad\qquad\qquad\qquad\qquad\cdot
  f_{i,j,k}\big(x^c_{i,j,k} - x^c_{i+1,j,k} - x^a_{i,j,k} + x^a_{i,j,k+1} + y^b_{i,j,k} \big) \\
  &\qquad\qquad\qquad\qquad\qquad\cdot
  f_{i,j,k}\big(x^b_{i,j,k} - x^b_{i,j,k+1} - x^c_{i,j,k} + x^c_{i,j+1,k} + y^a_{i,j,k} \big)
  \Big]
  \,\rd\bsx.
\end{align*}
For each $\bsh = (\bsh^a,\bsh^b,\bsh^c)\in \bbZ^{3L^3}$, we compute the
Fourier coefficient of $\calI(\bsy)$ to arrive at
\begin{align*}
  \widehat{\calI}(\bsh)
  &=
  \bigg[\prod_{i=0}^{L-1} \prod_{j=0}^{L-1} \prod_{k=0}^{L-1}
  \Big(\widehat{f_{i,j,k}}(h^a_{i,j,k})\,\widehat{f_{i,j,k}}(h^b_{i,j,k})\,\widehat{f_{i,j,k}}(h^c_{i,j,k})\Big)\bigg]
  \int_{D^{3L^3}} e^{2\pi\ri\, p(\bsh,\bsx)}\,\rd\bsx, \\
  p(\bsh,\bsx)
  &=
  \sum_{i=0}^{L-1} \sum_{j=0}^{L-1} \sum_{k=0}^{L-1} \Big[
  \Big(h^c_{i,j,k} - h^c_{i,j-1,k} - h^b_{i,j,k} + h^b_{i,j,k-1} \Big) \,x^a_{i,j,k} \\
  &
  \quad
  + \big(h^a_{i,j,k} - h^a_{i,j,k-1} - h^c_{i,j,k} + h^c_{i-1,j,k} \big) \,x^b_{i,j,k}
  + \big(h^b_{i,j,k} - h^b_{i-1,j,k} -  h^a_{i,j,k} + h^a_{i,j-1,k} \big) \,x^c_{i,j,k} \Big].
\end{align*}
We conclude that $\int_{D^{3L^3}} e^{2\pi\ri\, p(\bsh,\bsx)}\,\rd\bsx$
is equal to $1$ if and only if $\bsh$ belongs to a restricted index set
$\calH \subset\bbZ^{3L^3}$, satisfying \eqref{eq:system} for all indices
$i,j,k$ modulo $L$,
and the integral is equal to $0$ otherwise. We obtain the required formula
by taking $\calI(\bszero) = \sum_{\bsh\in\bbZ^{3L^3}}
\widehat{\calI}(\bsh)$.
\end{proof}

\bibliographystyle{plain}

\end{document}